\documentclass[12pt]{amsart}
\usepackage{latexsym}
\usepackage{amssymb}
\usepackage{amscd}
\usepackage{mathrsfs}
\usepackage[all]{xy}

\renewcommand\a{\alpha}
\renewcommand\b{\beta}
\newcommand\g{\gamma}
\renewcommand\d{\delta}
\newcommand\la{\lambda}
\newcommand\z{\zeta}

\newcommand\s{\sigma}

\newcommand\f{\phi}
\newcommand\vf{\varphi}

\renewcommand\t{\tau}

\newcommand{\vT}{\varTheta}

\newcommand\ve{\varepsilon}

\newcommand\Ql{\bar{\mathbf Q}_l}
\newcommand\BA{\mathbf A}
\newcommand\BP{\mathbf P}
\newcommand\BQ{\mathbf Q}
\newcommand\BF{\mathbf F}
\newcommand\BC{\mathbf C}
\newcommand\BD{\mathbf D}
\newcommand\BE{\mathbf E}
\newcommand\BR{\mathbf R}

\newcommand\BZ{\mathbf Z}

\newcommand\BN{\mathbf N}

\newcommand\BJ{\mathbf J}
\newcommand\bB{\mathbf B}

\newcommand\BU{\mathbf U}
\newcommand\BV{\mathbf V}

\newcommand\Bn{\mathbf n}

\newcommand\Be{\mathbf e}
\newcommand\Bf{\mathbf f}

\newcommand\Bk{\mathbf k}

\newcommand\CB{\mathcal{B}}
\newcommand\ZC{\mathcal{C}}

\newcommand\CQ{\mathcal{Q}}

\newcommand\CV{\mathcal{V}}

\newcommand\CJ{\mathcal{J}}

\newcommand\SC{\mathscr{C}}

\newcommand\SL{\mathscr{L}}

\newcommand\SP{\mathscr{P}}

\newcommand\Fg{\mathfrak g}

\newcommand\iv{^{-1}}

\newcommand\wt{\widetilde}

\newcommand\ol{\overline}
\newcommand\ul{\underline}

\newcommand\Hom{\operatorname{Hom}}

\newcommand\Ind{\operatorname{Ind}}

\newcommand\Res{\operatorname{Res}}
\newcommand\Mod{\operatorname{-Mod}}

\newcommand\gm{\operatorname{{\bf gm}}}
\newcommand\gp{\operatorname{{\bf gp}}}
\newcommand\gpp{\operatorname{-{\bf gp}}}
\newcommand\gmm{\operatorname{-{\bf gm}}}
\newcommand\qdim{\operatorname{qdim}}

\newcommand\Tr{\operatorname{Tr}\,}
\newcommand\ch{\operatorname{ch}}

\newcommand\id{\operatorname{id}}
\newcommand\Id{\operatorname{Id}}
\newcommand\lp{\operatorname{\langle}}
\newcommand\rp{\operatorname{\rangle}}
\newcommand\lpp{\operatorname{\lp\!\!\lp\,}}
\newcommand\rpp{\operatorname{\rp\!\!\rp}}

\newcommand\nat{^{\natural}}

\newcommand\weit{\operatorname{wt}}

\newcommand\soc{\operatorname{soc}}
\newcommand\head{\operatorname{hd}}

\newcommand{\isom}{\,\raise2pt\hbox{$\underrightarrow{\sim}$}\,}
\numberwithin{equation}{section}

\newtheorem{thm}{Theorem}[section]
\newtheorem{lem}[thm]{Lemma}
\newtheorem{cor}[thm]{Corollary}
\newtheorem{prop}[thm]{Proposition}

\def \para#1{\par\medskip\textbf{#1}
              \addtocounter{thm}{1}}

\def \remark#1{\par\medskip\noindent
                \textbf{Remark #1}
                \addtocounter{thm}{1}}

\def \remarks#1{\par\medskip\noindent
                \textbf{Remarks #1}
                \addtocounter{thm}{1}}

\begin{document}
\setlength{\baselineskip}{4.9mm}
\setlength{\abovedisplayskip}{4.5mm}
\setlength{\belowdisplayskip}{4.5mm}
\renewcommand{\theenumi}{\roman{enumi}}
\renewcommand{\labelenumi}{(\theenumi)}
\renewcommand{\thefootnote}{\fnsymbol{footnote}}
\renewcommand{\thefootnote}{\fnsymbol{footnote}}
\allowdisplaybreaks[2]
\parindent=20pt
\medskip
\begin{center}
{\bf Foldings of KLR  algebras}  
\par
\vspace{1cm}
Ying Ma, Toshiaki Shoji and Zhiping Zhou 
\\
\title{}
\end{center}

\begin{abstract}
Let $\BU_q^-$ be the negative half of the quantum group associated to the 
Kac-Moody algebra $\Fg$, and $\ul\BU_q^-$ the quantum group obtained 
by a folding of $\Fg$. Let $\BA = \BZ[q,q\iv]$. 
McNamara showed that $\ul\BU_q^-$ is categorified over a certain extension 
ring $\wt\BA$ of $\BA$, by using the folding theory of KLR algebras.
He posed a question whether $\wt\BA$ coincides with $\BA$ or not. 
In this paper, we give an affirmative answer for this problem. 
\end{abstract}

\maketitle
\pagestyle{myheadings}

\begin{center}
{\sc Introduction}
\end{center}

\par\medskip
Let $X = (I,(\ ,\ ))$ be a Cartan datum, 
and $\BU_q^-$ the negative half of the quantum group associated to $X$.  
Let $Q = \bigoplus_{i \in I}\BZ \a_i$ be the root lattice of $X$,
and set $Q_+ = \sum_{i \in I}\BN\a_i$.
Set $\BA = \BZ[q,q\iv]$, and let ${}_{\BA}\BU_q^-$ be Lusztig's integral
form of $\BU_q^-$, which is an $\BA$-subalgebra of $\BU_q^-$.  
In the case where $X$ is of symmetric type, 
Lusztig constructed in [L1] the canonical basis of $\BU_q^-$, by using 
the geometry of a quiver $\overrightarrow{Q}$ associated to $X$.  More precisely, 
he constructed, for each $\b \in Q_+$, a category $\CQ_{V_{\b}}$ consisting of 
certain semisimple $\Ql$-complexes on the representation space $V_{\b}$ of $\overrightarrow{Q}$ 
associated to $\b$, and showed that the direct sum  
$K(\CQ) = \bigoplus_{\b \in Q_+}K(\CQ_{V_{\b}})$ 
of the Grothendieck group $K(\CQ_{V_{\b}})$ of $\CQ_{V_{\b}}$ has a structure of 
$A$-algebra, which  
is isomorphic to ${}_{\BA}\BU_q^-$.  
The canonical basis of $\BU_q^-$ is obtained from a natural basis of 
$K(\CQ_{V_{\b}})$ coming from simple perverse sheaves in 
$\CQ_{V_{\b}}$. 
\par
Let $\s$ be an admissible diagram automorphism on $X$ (see 1.5), which 
gives a bijection $\s : I \to I$.
Let $J$ be the set of $\s$-orbits in $I$. 
Then $\ul X = (J, (\ ,\ ))$ gives a Cartan datum.
We denote by $\ul\BU_q^-$ the corresponding quantum group.   
In [L2], Lusztig constructed the canonical basis of $\ul\BU_q^-$, 
by extending the discussion in the symmetric case. 
By a suitable choice of the orientation, $\s$ acts on the quiver $\overrightarrow{Q}$, 
and it induces a periodic functor $\s^*$ on $\CQ_{V_{\b}}$ whenever $\b$ is $\s$-stable. 
He defined a new category $\wt\CQ_{V_{\b}}$ attached to $\CQ_{V_{\b}}$ and $\s^*$, 
and a (modified) Grotehndieck group $K(\wt\CQ_{V_{\b}})$. 
Let $Q_+^{\s}$ be the set of $\s$-stable elements in $Q_+$.  Then the direct sum 
$K(\wt\CQ) = \bigoplus_{\b \in Q_+^{\s}}K(\wt\CQ_{V_{\b}})$ has a structure of 
$\BZ[\z_{\Bn}, q, q\iv]$-algebra, where $\Bn$ is the order of $\s$, and $\z_{\Bn}$ 
is a primitive $\Bn$-th root of unity in $\Ql$. 
He proved that a certain $\BA$-subalgebra ${}_{\BA}K(\wt\CQ)$ is isomorphic to 
${}_{\BA}\ul\BU_q^-$.  In that case, a natural basis of ${}_{\BA}K(\wt\CQ)$ gives 
the canonical basis of $\ul\BU_q^-$, up to sign, i.e., gives the canonical signed basis. 
The canonical basis of $\ul\BU_q^-$ is obtained from it by suitably fixing the signs, which 
was done by an algebraic argument using Kashiwara's tensor product rule.  
\par
$K(\CQ)$ (for symmetric $X$) gives a geometric categorification of $\BU_q^-$.  
On the other hand,   
in [KL], [R1], Khovanov-Lauda, and Rouquier introduced the Khovanov-Lauda-Rouquier
algebra $R = \bigoplus_{\b \in Q_+}R(\b)$ (KLR algebra in short) over an algebraically 
closed field $\Bk$, and gave 
a categorification of $\BU_q^-$ (for general $X$) in terms of KLR algebras. 
For each $\b \in Q_+$, let 
$R(\b)\gpp$ be the category of finitely generated graded projective $R(\b)$-modules, 
and $K_{\gp}(\b)$ the Grothendieck group of $R(\b)\gpp$.  They proved that
$K_{\gp} = \bigoplus_{\b \in Q_+}K_{\gp}(\b)$ has a structure of $\BA$-algebra, and 
is isomorphic to ${}_{\BA}\BU_q^-$.    
\par
It is natural to expect a similar discussion as in Lusztig's case will hold 
also for KLR algebras. McNamara [M] developed the theory of foldings for
KLR algebras. Let $X$ and $\ul X$ be as above (for $X$: general type),
and choose $\b \in Q_+^{\s}$.     
By a suitable choice of defining parameters of $R$, $\s$ acts on $R(\b)$, thus
we obtain a periodic functor $\s^*$ on $R(\b)\gpp$. One can define a new category 
$\wt R(\b)\gpp = \SP_{\b}$, and a modified Grothendieck group $K(\SP_{\b})$. 
$K(\SP) = \bigoplus_{\b \in Q_+^{\s}}K(\SP_{\b})$ has a structure of  
$\BZ[\z_{\Bn}, q, q\iv]$-algebra. McNamara constructed a natural basis $\ul\bB$ 
of $K(\SP)$ by using the theory of crystals.  Define an extension $\wt\BA$ of $\BA$ 
as the smallest 
subring $\BZ[\z_{\Bn},q,q\iv]$ containing the structure constants of 
$\ul\bB$ with respect to the product in $K(\SP)$. Hence the $\wt\BA$-span ${}_{\wt\BA}K(\SP)$ 
of $\ul\bB$ gives an $\wt\BA$-algebra.  He proved that ${}_{\wt\BA}K(\SP)$ is 
isomorphic to ${}_{\wt\BA}\ul\BU_q^-$, where ${}_{\wt\BA}\ul\BU_q^-$ is the extension 
of ${}_{\BA}\ul\BU_q^-$ to $\wt\BA$, thus obtained a categorification of $\ul\BU_q^-$.
He showed that $\BA \subset \wt\BA \subset \BZ[\z_{\Bn} + \z_{\Bn}\iv, q,q\iv]$, 
and posed a question whether $\wt\BA$ coincides with $\BA$ or not ([M, Question 11.2]). 
\par
The aim of this paper is to show that this certainly holds (Theorem  3.16). 
The main ingredient for the proof is the following result proved in [MSZ].
We consider the case where the order $\Bn$ of $\s$ is a power of a prime number $\ell$. 
Let $\BF = \BZ/\ell\BZ$ be the finite field of $\ell$-elements. Let 
$\BA' = \BF[q, q\iv] = \BA/\ell \BA$, and consider the $\BA'$-algebras
${}_{\BA'}\BU_q^-$ and ${}_{\BA'}\ul\BU_q^-$, by changing the base ring from $\BA$ 
to $\BA'$.  $\s$ acts naturally on ${}_{\BA'}\BU_q^-$, and let ${}_{\BA'}\BU_q^{-,\s}$
be the fixed point subalgebra by $\s$. In [MSZ, Thm. 3.4, Thm. 4.27], it is proved 
that ${}_{\BA'}\ul\BU_q^-$ is realized as the quotient algebra of ${}_{\BA'}\BU_q^{-,\s}$, 
and that the canonical basis of $\ul\BU_q^-$ is obtained from the $\s$-stable canonical 
basis of $\BU_q^-$, up to sign, through the quotient map.  
This gives an alternate approach for Lusztig's geometric construction of canonical signed 
basis of $\ul\BU_q^-$.  A similar argument also works for the KLR algebras. In this scheme, 
the finite field $\BF$ enters in the theory, but $\z_{\Bn}$ is not involved.     
By comparing this with McNamara's isomorphism, we obtain the required result. 
\par
In the case where $X$ is symmetric, and $R$ is a symmetric type KLR algebra (see 5.1) 
over a field $\Bk$ of characteristic zero, Varagnolo-Vasserot [VV] and Rouquier [R2] proved 
that the natural basis of $K_{\gp}$ gives the canonical basis (which coincides with 
Kashiwara's lower global basis by [GL]) of $\BU_q^-$, by using 
the geometric interpretation of $R$ in terms of the Ext algebra coming from  Lusztig's 
category $\CQ_{V_{\b}}$.  A similar result holds also for $K(\SP)$, namely, 
the basis $\ul\bB$ corresponds, via McNamara's isomorphism, to the canonical basis 
of $\ul\BU_q^-$, which coincides with the lower global basis.  This is stated in 
the last paragraph of [M], with a brief indication for the proof, which is based  
on the geometric argument.  In Theorem 5.13 and Theorem 5.16, 
we give a simple proof of this fact, without using the geometry 
(though assuming the results in [VV], [R2]).       

\section{ Quantum groups }

\para{1.1.}
Let $X = (I, (\ ,\ ))$ be a Cartan datum, where $I$ is a finite set, 
and $(\ ,\ )$ is a symmetric bilinear form on the vector space 
$\bigoplus_{i \in I}\BQ \a_i$ with basis $\a_i$, satisfying the properties 
that 
\begin{enumerate}
\item 
$(\a_i,\a_i) \in 2\BZ_{> 0}$ for any $i \in I$, 
\item $\frac{2(\a_i,\a_{i'})}{(\a_i,\a_i)} \in \BZ_{\le 0}$ 
for any $i \ne i' \in I$.
\end{enumerate}

For $i, i' \in I$, set $a_{ii'} = 2(\a_i, \a_{i'})/(\a_i, \a_i) \in \BZ$. 
$A = (a_{ii'})$ is called the Cartan matrix associated to $X$.  
The Cartan datum is called symmetric if $(\a_i, \a_i) = 2$ for any $i \in I$. 
\par
Let $Q = \bigoplus_{i \in I}\BZ \a_i$ be the root lattice of $X$. 
We set $Q_+ = \sum_{i \in I}\BN \a_i$, and $Q_- = -Q_+$. 

\para{1.2.}
Let $q$ be an indeterminate, and for an integer $n$, a positive integer $m$, 
set
\begin{equation*}
[n] = \frac{q^n - q^{-n}}{q -q\iv}, \quad [m]^! = [1][2]\cdots [m], 
    \quad [0]^! = 1.
\end{equation*}
\par
For each $i \in I$, set $d_i = (\a_i,\a_i)/2 \in \BN$, and 
$q_i = q^{d_i}$. 
We  denote by $[n]_i$ the element obtained from $[n]$ by 
replacing $q$ by $q_i$.  
\par
For the Cartan datum $X$, let $\BU_q^- = \BU_q^-(X)$ be the negative 
half of the quantum group associated to $X$. $\BU_q^-$ is an associative 
algebra over $\BQ(q)$ generated by $f_i \ (i \in I)$ satisfying the $q$-Serre relations 
\begin{equation*}
\tag{1.2.1}
\sum_{k + k' = 1 - a_{ii'}}(-1)^kf_i^{(k)}f_{i'}f_i^{(k')} = 0,
\quad\text{ for $i \ne i' \in I$,}
\end{equation*}
where $f_i^{(n)} = f_i^n/[n]_i^!$.
\par
\par   
Set $\BA = \BZ[q,q\iv]$, and let ${}_{\BA}\BU_q^-$ be Lusztig's integral form 
of $\BU_q^-$, namely, the $\BA$-subalgebra of $\BU_q^-$ generated by 
$f_i^{(n)}$ for $i \in I, n \in \BN$.   
\par
We define a $\BQ$-algebra automorphism, called the bar-involution, 
${}^-: \BU_q^- \to \BU_q^-$ by $\ol q = q\iv$, $\ol f_i = f_i$ for $i \in I$.
We define an anti-algebra automorphism ${}^* : \BU_q^- \to \BU_q^-$
by $f_i^* = f_i$ for any $i \in I$.  

\para{1.3.}
$\BU_q^-$ has a weight space decomposition 
$\BU_q^- = \bigoplus_{\b \in Q_-}(\BU_q^-)_{\b}$, where 
$(\BU_q^-)_{\b}$ is a subspace of $\BU_q^-$ spanned by elements $f_{i_1}\cdots f_{i_N}$
such that $\a_{i_1} + \cdots + \a_{i_N} = -\b$. 
$x \in \BU_q^-$ is called homogeneous with $\weit x = \b$ if $x \in (\BU_q^-)_{\b}$. 
We define a multiplication on $\BU_q^- \otimes \BU_q^-$ by 

\begin{equation*}
\tag{1.3.1}
(x_1\otimes x_2)\cdot (x_1'\otimes x_2') = q^{-(\weit x_2, \weit x_1')}
                                             x_1x_1'\otimes x_2x_2' 
\end{equation*}
where $x_1, x_1', x_2, x_2'$ are homogeneous in $\BU_q^-$. 
Then $\BU_q^-\otimes \BU_q^-$ becomes an associative algebra with respect to
this twisted product. One can define a homomorphism 
$r : \BU_q^- \to \BU_q^-\otimes \BU_q^-$ by $r(f_i) = f_i\otimes 1 + 1 \otimes f_i$ 
for each $i \in I$. It is known 
that there exists a unique bilinear form 
$(\ ,\ )$ on $\BU_q^-$ satisfying the following properties; 
$(1,1) = 1$ and

\begin{equation*} 
\tag{1.3.2}
\begin{aligned}
(f_i, f_j) = \d_{ij}(1 - q_i^2)\iv, \\
(x, y'y'') = (r(x), y'\otimes y''), \\
(x'x'', y) = (x'\otimes x'', r(y)),
\end{aligned}
\end{equation*}
where the bilinear form on $\BU_q^-\otimes \BU_q^-$ is
defined by $(x_1\otimes x_2, x_1'\otimes x_2') = (x_1, x_1')(x_2, x_2')$. 
Thus defined bilinear form is symmetric and non-degenerate.
\par
Following [L2, 1.2.13], 
for each $i$, we define a $\BQ(q)$-linear map ${}_ir : \BU_q^- \to \BU_q^-$
by the condition that  
\begin{equation*}
\tag{1.3.3}
r(x) = f_i\otimes {}_ir(x) + \sum y\otimes z,
\end{equation*}
 where $y$ runs over homogeneous
elements such that $\weit y \ne -\a_i$. 
Since ${}_ir$ coincides with the operator $e_i'$ defined in 
[K1, 3.4], (see [L2, Prop. 3.1.6] and [K1, Lemma 3.4.1]),  
hereafter we write ${}_ir$ as $e_i'$. 
Thus, if we define the action of $f_i$ on $\BU_q^-$ by the left 
multiplication, then $e_i', f_{i'}$ satisfies the $q$-boson relations, as operators 
on $\BU_q^-$,  
\begin{equation*}
\tag{1.3.4}
e_i'f_{i'} = q^{-(\a_i, a_{i'})}f_{i'}e_i' + \d_{ii'}.
\end{equation*}

\para{1.4.}
Let $V$ be a $\BQ(q)$-subspace of $\BU_q^-$.
A basis $\CB$ of $V$ is said to be almost orthonormal if
\begin{equation*}
\tag{1.4.1}
(b, b') \in \begin{cases}
               1 + q\BZ[[q]] \cap \BQ(q) &\quad\text{ if } b = b', \\
               q\BZ[[q]] \cap \BQ(q)     &\quad\text{ if } b \ne b'.
             \end{cases}
\end{equation*}
\par
Recall that $\BA = \BZ[q,q\iv]$.  Let $\BA_0 = \BQ[[q]] \cap \BQ(q)$. 
Set
\begin{equation*}
\tag{1.4.2}
\SL_{\BZ}(\infty) = \{ x \in {}_{\BA}\BU_q^- \mid (x,x) \in \BA_0\}.
\end{equation*}
Then $\SL_{\BZ}(\infty)$ is a $\BZ[q]$-submodule of ${}_{\BA}\BU_q^-$. 
It is known that if $\CB$ is an $\BA$-basis of ${}_{\BA}\BU_q^-$, which is almost orthonormal, 
then $\CB$ gives a $\BZ[q]$-basis of $\SL_{\BZ}(\infty)$ by [L2, Lemma 16.2.5].  
\par
We define a subset $\wt\CB$ of $\BU_q^-$ by 
\begin{equation*}
\tag{1.4.3}
\wt\CB = \{ x \in {}_{\BA}\BU_q^- \mid  \ol x = x, (x, x) \in 1 + q\BZ[[q]]\}. 
\end{equation*}  
If there exists a basis $\CB'$ of $\BU_q^-$ such that 
$\wt\CB = \CB' \sqcup -\CB'$, then $\wt\CB$ is called the canonical signed basis.  
In that case, the basis $\CB'$ is determined uniquely, up to sign, by the condition
(1.4.3). 
Lusztig proved in [L2], by using the geometric theory of quivers, that 
$\wt\CB$ is the canonical signed basis. Kashiwara's global crystal basis 
$\CB$ in [K1] also satisfies the condition that $\wt\CB = \CB \sqcup -\CB$ 
(see [K1, Prop. 5.1.2], but note that the inner products $(\ ,\ )$ on $\BU_q^-$ 
defined by Lusztig and by Kashiwara differ by a scalar factor on each weight space).

\para{1.5.}
A permutation $\s : I \to I$ is called an admissible automorphism on $X$ if
it satisfies the property that  $(\a_i, \a_{i'}) = (\a_{\s(i)}, \a_{\s(i')})$
for any $i,i' \in I$, and that $(\a_i, \a_{i'}) = 0$ if $i$ and $i'$ belong 
to the same $\s$-orbit in $I$. 
Assume that $\s$ is admissible. We denote by $\Bn$ the order of $\s : I \to I$. 
Let $J$ be the set of $\s$-orbits in $I$. 
For each $j \in J$, set $\a_j = \sum_{i \in j}\a_i$, and consider the subspace
$\bigoplus_{j \in J}\BQ \a_j$ of $\bigoplus_{i \in I}\BQ \a_i$, with basis $\a_j$.
We denote by $|j|$ the size of the orbit $j$ in $I$. 
The restriction of the form $(\ ,\ )$ on $\bigoplus_{j \in J}\BQ \a_j$ 
is given by 
\begin{equation*}
(\a_j, \a_{j'})  = \begin{cases} 
                    (\a_i,\a_i)|j|  \quad (i \in j) &\quad\text{ if } j = j', \\
                    \sum_{i \in j, i' \in j'}(\a_i, \a_{i'})
                           &\quad\text{ if } j \ne j'.  
                   \end{cases}
\end{equation*}  
Then $\ul X = (J, (\ ,\ ))$ turns out to be a Cartan datum, which is called the Cartan 
datum induced from $(X, \s)$. $\s$ acts naturally on $Q$, and the set $Q^{\s}$
of $\s$-fixed elements in $Q$ is identified with the root lattice of $\ul X$. 
We have $Q_+^{\s} = \sum_{j \in J}\BN \a_j$. 

\par
Let $\s$ be an admissible automorphism of $X$.  Then $\s$ induces 
an algebra automorphism on $\BU_q^- \to \BU_q^-$ 
by $f_i \mapsto f_{\s(i)}$, which we also denote by $\s$. 
The action of $\s$ leaves ${}_{\BA}\BU_q^-$ invariant.  We denote by
${}_{\BA}\BU_q^{-,\s}$ the fixed point subalgebra of ${}_{\BA}\BU_q^-$.  
\par
Let $\ul X$ be the Cartan datum induced from $(X, \s)$. 
Let $\ul \BU_q^- = \BU_q^-(\ul X)$ the quantum group associated to $\ul X$. 
Thus $\ul\BU_q^-$ is an associative $\BQ(q)$-algebra, with generators 
$f_j$ ($j \in J$), and similar relations as in (1.2.1).
The integral form ${}_{\BA}\ul\BU_q^-$ of $\ul\BU_q^-$ is the $\BA$-subalgebra
generated by $f_j^{(n)}$ ($j \in J, n \in \BZ$).  

\para{1.6.}{\bf Quotient algebras $\BV_q$.} 
In general, let $V$ be a vector space on which $\s$ acts. For each $x \in V$, 
we denote by $O(x)$ the orbit sum of $x$, namely 
$O(x) = \sum_{0 \le i < k}\s^i(x)$, where $k$ is the smallest integer such that 
$\s^k(x) = x$. Hence $O(x)$ is a $\s$-invariant element in $V$.
\par
Let $\s : X \to X$ be as in 1.5, and   
$\Bn$ the order of $\s : I \to I$. In the rest of this section, we assume that
$\Bn$ is a power of a prime number $\ell$.  Let $\BF = \BZ/\ell\BZ$ be the finite field 
of $\ell$ elements, and set $\BA' = \BF[q,q\iv] = \BA/\ell\BA$. 
We consider the $\BA'$-algebra
\begin{equation*}
{}_{\BA'}\BU_q^{-,\s} = \BA'\otimes_{\BA}{}_{\BA}\BU_q^{-,\s}
                     \simeq {}_{\BA}\BU_q^{-,\s}/\ell({}_{\BA}\BU_q^{-,\s}). 
\end{equation*}  
\par  
Let $\BJ$ be an $\BA'$-submodule of ${}_{\BA'}\BU_q^{-, \s}$ generated by 
$O(x)$ for $x \in {}_{\BA'}\BU_q^-$ such that $\s(x) \ne x$.  
Then $\BJ$ is a two-sided ideal of ${}_{\BA'}\BU_q^{-,\s}$. 
We define an $\BA'$-algebra $\BV_q$ as the quotient algebra of ${}_{\BA'}\BU_q^{-,\s}$, 
\begin{equation*}
\tag{1.6.1}
\BV_q = {}_{\BA'}\BU_q^{-,\s}/\BJ.
\end{equation*}
Let $\pi : {}_{\BA'}\BU_q^{-,\s} \to \BV_q$ be the natural projection.  
\par 
For each $j \in J$ and $a \in \BN$, set 
$\wt f_j^{(a)} = \prod_{i \in j}f_i^{(a)}$. 
Since $\s$ is admissible, $f_i^{(a)}$ and $f_{i'}^{(a)}$ commute each other 
for $i,i' \in j$. Hence $\wt f_j^{(a)}$ does not depend on the order of the product, 
and $\wt f^{(a)}_j \in {}_{\BA}\wt\BU_q^{-,\s}$. We denote its image 
in ${}_{\BA'}\BU_q^{-,\s}$ also by $\wt f^{(a)}_j$.  
Thus we can define $g_j^{(a)} \in \BV_q$ by 
\begin{equation*}
\tag{1.6.2}
g^{(a)}_j = \pi(\wt f_j^{(a)}). 
\end{equation*}
In the case where $a = 1$, we set 
$\wt f_j^{(1)} = \wt f_j = \prod_{i \in j}f_i$ 
and $g^{(1)}_j = g_j$. 
\par
We define an ${\BA'}$-algebra ${}_{\BA'}\ul\BU_q^-$ by 
${}_{\BA'}\ul\BU_q^- = \BA'\otimes_{\BA}{}_{\BA}\ul\BU_q^-$. 
\par
The following result was proved in [MSZ].

\begin{thm}[{[MSZ, Thm. 3.4]}]  
The assignment $f_j^{(a)} \mapsto g_j^{(a)}$ ($j \in J$) 
gives an isomorphism $\Phi : {}_{\BA'}\ul\BU_q^- \isom \BV_q$
of $\BA'$-algebras. 
\end{thm}

\remark{1.8} In Theorem 3.4 in [MSZ],  the existence of the canonical 
signed basis of $\BU_q^-$ is assumed.  This certainly holds, by 1.4. 
In [MSZ, Thm. 4.27], 
the canonical signed basis for $\ul\BU_q^-$ 
was constructed from the canonical basis in the symmetric case, 
by making use of Theorem 1.7.  

\para{1.9.}{\bf The dual space $(\BU_q^-)^*$.} 
Let $\BU_q^- = \bigoplus_{\b \in Q^-}(\BU_q^-)_{\b}$ be the weight space 
decomposition of $\BU_q^-$.
We define the graded dual $(\BU_q^-)^*$ of $\BU_q^-$ by 
\begin{equation*}
(\BU_q^-)^* = \bigoplus_{\b \in Q^-}(\BU_q^-)_{\b}^*,
\end{equation*}
where $(\BU_q^-)_{\b}^* = \Hom_{\BQ(q)}((\BU_q^-)_{\b}, \BQ(q))$. 
$f_i^{(n)}$ and $e_i'$ act on $\BU_q^-$, and we define the actions of 
${e_i'}^{(n)}$ and $f_i$ on $(\BU_q^-)^*$ as the transpose of $f_i^{(n)}, e_i'$. 
We write $e_i' = e_i'^{(n)}$ for $n = 1$ (as operators on $(\BU_q^-)^*$). 
\par
${}_{\BA}\BU^-_q$ is decomposed as 
\begin{equation*}
_{\BA}\BU^-_q = \bigoplus_{\b \in Q_-}{}_{\BA}(\BU_q^-)_{\b},
\end{equation*}
where ${}_{\BA}(\BU_q^-)_{\b} = {}_{\BA}\BU_q^- \cap (\BU_q^-)_{\b}$
is a free $\BA$-module. 
We define $_{\BA}(\BU_q^-)^*$ by 
\begin{equation*}
_{\BA}(\BU_q^-)^* = \bigoplus_{\b \in Q_-} {}_{\BA}(\BU_q^-)^*_{\b}, 
\end{equation*}
where $_{\BA}(\BU_q^-)_{\b}^* = \Hom_{\BA}({}_{\BA}(\BU_q^-)_{\b}, \BA)$.
Since $\BQ(q)\otimes_{\BA}{}_{\BA}(\BU_q^-)^* \simeq (\BU_q^-)^*$, 
$_{\BA}(\BU_q^-)^*$ is regarded as an $\BA$-submodule of $(\BU_q^-)^*$,  
and $_{\BA}(\BU_q^-)^*$ gives an $\BA$-lattice of $(\BU_q^-)^*$.
\par
The actions of $f_i^{(n)}, e_i'$ preserve $_{\BA}\BU_q^-$, and 
the actions of ${e_i'}^{(n)}, f_i$ preserve $_{\BA}(\BU_q^-)^*$. 
\par
Recall that $u \mapsto \ol u$ is the bar involution on $\BU_q^-$.
We define the bar involution $\vf \mapsto \ol\vf$ on $(\BU_q^-)^*$ by 
$\ol\vf(\ol u) = \ol{\vf(u)}$ ($\vf \in (\BU_q^-)^*, u \in \BU_q^-$). 

In [K2], Kashiwara introduced the upper global basis $\CB^*$ of $(\BU_q^-)^*$, which 
is characterized by the following properties. 

\begin{thm}  
There exists a unique $\BA$-basis $\CB^*$ of ${}_{\BA}(\BU_q^-)^*$ 
satisfying the following properties. For $b \in \CB^*$, set
\begin{equation*}
\ve_i(b) = \max\{ k \ge 0 \mid e_i'^kb \ne 0 \}.
\end{equation*}
\begin{enumerate}
\item  \ $\CB^*$ has a weight space decomposition 
$\CB^* = \bigsqcup_{\b \in Q_-}\CB^*_{\b}$, where 
$\CB^*_{\b} = \CB^* \cap {}_{\BA}(\BU_q^-)_{\b}^*$.  

\item \ $1^* \in \CB^*$, where $\{ 1^* \} \subset {}_{\BA}(\BU_q^-)^*_0$ 
is the dual basis of $\{ 1\} \subset {}_{\BA}(\BU_q^-)_0$. 

\item \ Assume that $e_i'b \ne 0$.  Then there exists 
a unique $\wt\Be_ib \in \CB^*$ satisfying the following. 
\begin{equation*}
e_i'b = [\ve_i(b)]_i\wt\Be_ib + 
  \sum_{\substack{b' \in \CB^* \\ \ve_i(b') < \ve_i(b)-1}}
     E_{bb'}^{(i)}b' \qquad (E_{bb'}^{(i)} \in qq_i^{1 - \ve_i(b)}\BZ[q]).        
\end{equation*} 

\item \ There exists a unique $\wt\Bf_ib \in \CB^*$ 
satisfying the following.
\begin{equation*}
f_ib = q_i^{-\ve_i(b)}\wt\Bf_ib + 
       \sum_{\substack{ b' \in \bB^* \\ \ve_i(b') < \ve_i(b) + 1}}F^{(i)}_{bb'}b'
      \qquad (F^{(i)}_{bb'} \in qq_i^{-\ve_i(b)}\BZ[q]).
\end{equation*}

\item \ $\wt\Be_i\wt\Bf_i b = b$. 

\item \ $\wt\Bf_i\wt\Be_i b = b$ if $e_i'b \ne 0$. 

\item \ $\ol b = b$. 
\end{enumerate}
\end{thm}

The $\BA$-basis $\CB$ of ${}_{\BA}\BU_q^-$ is defined as the dual basis 
of $\CB^*$.  $\CB$ is called the lower global basis. $\CB$ coincides with 
the global crystal basis given in 1.4.

\para{1.11.}{\bf The dual module $\BV_q^*$.} 
We define ${}_{\BA'}(\BU_q^-)^*$ similarly to ${}_{\BA'}\BU_q^-$.
Then $e'^{(n)}_i$ act on ${}_{\BA'}(\BU_q^-)^*$. 
\par
We define an action of $\s$ on $(\BU_q^-)^*$ by 
$(\s \vf)(x) = \vf(\s\iv(x))$ for $\vf \in (\BU_q^-)^*$ and $x \in \BU_q^-$.
Thus the natural pairing $\lp\ ,\ \rp : \BU_q^- \times (\BU_q^-)^* \to \BQ(q)$
satisfies the property $\lp \s(x), \s(\vf)\rp = \lp x, \vf\rp$. 
By the above discussion, the pairing $\lp\ ,\ \rp$ induces a perfect pairing 
$\lp\ ,\ \rp : {}_{\BA'}\BU_q^- \times {}_{\BA'}(\BU_q^-)^* \to \BA'$, 
which is compatible with the action of $\s$. 
\par
By Theorem 1.10, $\s$ acts on $\CB^*$ as a permutation, 
hence acts on $\CB$ similarly. 
\par
Let ${}_{\BA'}(\BU_q^-)^{*,\s}$ be the submodule of ${}_{\BA'}(\BU_q^-)^*$
of $\s$-fixed elements. 
Let $\BJ^*$ be an $\BA'$-submodule of ${}_{\BA'}(\BU_q^-)^{*,\s}$ 
consisting of $O(x)$ for $x \in {}_{\BA'}(\BU_q^-)^*$ such that $O(x) \ne x$. 
We define $\BV_q^*$ as the quotient $\BA'$-module
\begin{equation*}
\tag{1.11.1}
\BV_q^* = {}_{\BA'}(\BU_q^-)^{*,\s}/\BJ^*.
\end{equation*}
\par
Since 
$\lp {}_{\BA}\BU_q^{-,\s}, \BJ^* \rp 
      = \lp \BJ, {}_{\BA'}(\BU_q^-)^{*,\s} \rp = 0$, 
the pairing $\lp\ ,\ \rp$ induces a pairing 
$\lp\ ,\ \rp : \BV_q \times \BV_q^* \to \BA'$. 
Let $\CB^{\s}$ be the set of $\s$-fixed elements in $\CB$, and 
$\CB^{*,\s}$ the set of $\s$-fixed elements in $\CB^*$. Then 
the image of $\CB^{\s}$ gives a basis of $\BV_q$, and the image of 
$\CB^{*,\s}$ gives a basis of $\BV_q^*$.  
Hence the pairing $\lp\ ,\ \rp : \BV_q \times \BV_q^* \to \BA'$ gives 
a perfect pairing on $\BA'$. 
\par
For each $j \in J$, we define an action of $\wt f_j^{(n)}$ on $\BV_q$ by 
the left multiplication by $g_j^{(n)} = \pi(f_j^{(n)})$ (see (1.6.2)).  
We define an operator $\wt e'^{(n)}_j$ on 
${}_{\BA'}(\BU_q^-)^*$ 
by $\wt e'^{(n)}_j = \prod_{i \in j}e'^{(n)}_i$. 
Then $\wt e'^{(n)}_j$ commutes with $\s$, and induces an action on 
${}_{\BA'}(\BU_q^-)^{*, \s}$. Since $\wt e'^{(n)}_j$ preserves $\BJ^*$, 
$\wt e'^{(n)}_j$ acts on $\BV_q^*$, which we also denote by $\wt e'^{(n)}_j$. 
$\wt e'^{(n)}_j$ coincides with the transpose of the operation $\wt f_j^{(n)}$ on 
$\BV_q$.  
\par
The operators $e'^{(n)}_j$ on ${}_{\BA'}(\ul\BU_q^-)^*$ is defined 
as in the case of ${}_{\BA'}(\BU_q^-)^*$. 
The following result is immediate from  Theorem 1.7.

\begin{prop}  
Let $\Phi^* : \BV_q^* \isom {}_{\BA'}(\ul\BU_q^-)^*$ be the $\BA'$-module isomorphism 
obtained as the transpose 
of the map $\Phi : {}_{\BA'}\ul\BU_q^- \isom \BV_q$.  
Then $\Phi^*$ commutes with the actions $\wt e'^{(n)}_j$  
on $\BV_q^*$ and $e'^{(n)}_j$ on ${}_{\BA'}(\ul\BU_q^-)^*$. 
\end{prop}

\par\bigskip
\section{ KLR algebras }

\para{2.1.}
Let $X = (I, (\,\ ))$ be a Cartan datum, and $\s$ an admissible automorphism on $X$. 
Assume that $\Bk$ is an algebraically closed field.  We define 
a KLR algebra associated to $(X, \s)$ as follows. For each $i, i' \in I$, 
choose  a polynomial $Q_{i,i'}(u,v) \in \Bk[u,v]$ 
satisfying the following properties;
\begin{enumerate}
\item \ $Q_{i,i}(u,v) = 0$,  
\item  \ $Q_{i,i'}(u,v)$ is written as 
$Q_{i,i'}(u,v) = \sum_{p,q}t_{ii';pq}u^pv^q$, where the 
coefficients $t_{ii';pq} \in \Bk$ satisfy the condition 
\begin{equation*}
(\a_i,\a_i)p + (\a_{i'},\a_{i'})q = -2(\a_i,\a_{i'}) 
        \quad\text{ if $t_{ii';pq} \ne 0$. }
\end{equation*}
Moreover, $t_{ii';-a_{ii'},0} \ne 0, t_{ii';0,-a_{ii'}} \ne 0$.  
\item  \ $Q_{i,i'}(u,v) = Q_{i',i}(v,u)$, 
\item \ $Q_{\s(i),\s(i')}(u,v) = Q_{i,i'}(u,v)$.  
\end{enumerate}
\par
For $\b = \sum_{i \in I}n_i\a_i \in Q_+$, we set $|\b| = \sum_in_i$.
For $\b \in Q_+$  such that $|\b| = n$, define $I^{\b}$ by 
\begin{equation*}
I^{\b} = \{ (i_1, \dots, i_n) \in I^n \mid \sum_{1 \le k \le n}\a_{i_k} = \b\}.
\end{equation*}
The KLR algebra $R(\b)$ is an associative $\Bk$-algebra 
defined by the generators, $e(\nu)$ ($\nu = (\nu_1, \dots, \nu_n) \in I^{\b}$), 
$x_k$ $(1 \le k \le n)$, $\tau_k \ (1 \le k < n)$,  with relations
\begin{align*}
\tag{1}
e(\nu)e(\nu') &= \d_{\nu,\nu'}e(\nu), \quad \sum_{\nu \in I^{\b}}e(\nu) = 1, \\
\tag{2}
x_kx_l &= x_lx_k, \quad  x_ke(\nu) = e(\nu)x_k, \\ 
\tag{3}
\tau_ke(\nu) &= e(s_k\nu)\tau_k, \quad 
           \tau_k\tau_l = \tau_l\tau_k, \quad\text{ if } |k - l| > 1, \\
\tag{4}
\tau_k^2e(\nu) &= Q_{\nu_k, \nu_{k+1}}(x_k, x_{k+1})e(\nu), \\
\tag{5}
(\tau_k x_l &- x_{s_k l}\tau_k)e(\nu) = 
                    \begin{cases}
                      -e(\nu)  &\quad\text{ if } l = k, \nu_k = \nu_{k+1}, \\
                       e(\nu)  &\quad\text{ if } l = k+1, \nu_k = \nu_{k+1}, \\
                       0       &\quad\text{ otherwise, }  
                    \end{cases}  \\
\tag{6}
(\tau_{k+1}&\tau_{k}\tau_{k+1} - \tau_{k}\tau_{k+1}\tau_{k})e(\nu)  \\
        &= \begin{cases}
     \displaystyle
     \frac{Q_{\nu_k,\nu_{k+1}}(x_k, x_{k+1}) - Q_{\nu_k,\nu_{k+1}}(x_{k+2}, x_{k+1})}
          {x_k - x_{k+2}}e(\nu)  &\quad\text{ if } \nu_k = \nu_{k+2}, \\
     0                           &\quad\text{ otherwise. }
          \end{cases}
\end{align*}  
(Here $s_k$ ($1 \le k < n$) is a transvection $(k, k+1)$ in the symmetric group
$S_n$. $S_n$ acts on $I^n$ by 
$w : (\nu_1, \dots, \nu_n) \mapsto (\nu_{w\iv(1)}, \dots, \nu_{w\iv (n)})$.)
\par
We define the algebra $R_n = \bigoplus_{|\b| = n}R(\b)$, and the algebra $R$ by
\begin{equation*}   
R = \bigoplus_{n \ge 0}R_n = \bigoplus_{\b \in Q_+}R(\b).
\end{equation*}
Note that the condition (iv) is used in considering the $\s$-setup for 
KLR algebras. The discussion until 2.9 is independent of $\s$, hence the condition 
(iv) is redundant.     
\par
The algebra $R(\b)$ is a $\BZ$-graded algebra, where the degree is 
defined, for $\nu = (\nu_1, \dots, \nu_n) \in I^{\b}$, 
 by 
\begin{equation*}
\deg e(\nu) = 0, \quad \deg x_ke(\nu) = (\a_{\nu_k}, \a_{\nu_k}), \quad 
\deg\tau_ke(\nu) = -(\a_{\nu_k}, \a_{\nu_{k+1}}).
\end{equation*}  
\par
For a graded $R(\b)$-module $M = \bigoplus_{i \in \BZ}M_i$, 
the grading shift by $-1$ is denoted by $qM$, 
namely $(qM)_i = M_{i-1}$. (Here $q$ is the indeterminate given in Section 1).  
There exists an anti-involution $\psi : R(\b) \to R(\b)$, which 
leaves all the generators $e(\nu), x_k$ and $\tau_k$ invariant. 

\para{2.2}
For $m, n \in \BN$,  there exists a $\Bk$-bilinear map
$R_m \times R_n \to R_{m+n}$ defined by 
\begin{align*}
&(x_k,1) \mapsto x_k, \quad (\tau_k,1) \mapsto \tau_k, 
\quad (1, x_k) \mapsto x_{m + k},  \quad (1, \tau_k) \mapsto \tau_{m+k}, \\ 
&(e(\nu), e(\nu')) \mapsto e(\nu,\nu'),
\end{align*} 
where for $\nu \in I^m, \nu' \in I^n$, we write their juxtaposition 
by $(\nu, \nu') \in I^{m+n}$. 
This defines an injective $\Bk$-algebra homomorphism 
$R_m \otimes R_n \to R_{m+n}$. By this map, 
$R_{m+n}$ has a structure of a right $R_m\otimes R_n$-module. 
\par
For an $R_m$-module $M$ and an $R_n$-module $N$, we define an $R_{m +n}$-module
$M\circ N$ by 
\begin{equation*}
M\circ N = R_{m +n}\otimes_{R_m\otimes R_n}M \otimes N.
\end{equation*} 
$M\circ N$ is called the convolution product of $M$ and $N$. 
Assume that $M$ is an $R(\b)$ module, and $N$ is an $R(\g)$-module 
with $|\b| = m, |\g| = n$.  Then $M$ is an $R_m$-module by the projection 
$R_m \to R(\b)$, and similarly, $N$ is an $R_n$-module. The 
$R_{m+n}$-module $M\circ N$ has the property that 
$e(\b + \g)(M \circ N) = M\circ N$, hence $M\circ N$ is regarded as an 
$R(\b + \g)$-module. The $R(\b+\g)$-module $M\circ N$ is defined directly as
\begin{equation*}
M\circ N = R(\b + \g)e(\b, \g)\otimes_{R(\b)\otimes R(\g)}M \otimes N,   
\end{equation*} 
where 
\begin{equation*}
e(\b,\g) = \sum_{\substack{ \nu \in I^{\b}, \nu' \in I^{\g}}}e(\nu,\nu').
\end{equation*}
\par
Let $\b,\g \in Q_+$. For an $R(\b + \g)$-module $M$, 
$e(\b,\g)M$ has a natural structure of $R(\b)\otimes R(\g)$-module.  

\para{2.3.}
Let $R(\b)\Mod$ be the abelian category of graded $R(\b)$-modules. 
Let $R(\b)\gmm$ be the full subcategory of $R(\b)\Mod$ 
consisting of finite dimensional graded $R(\b)$-modules,  
and $R(\b)\gpp$ the full subcategory of $R(\b)\Mod$ 
consisting of finitely generated graded projective 
$R(\b)$-modules.  Thus $R(\b)\gmm$ is an abelian category, and 
$R(\b)\gpp$ is an additive category.  
Let $K_{\gm}(\b)$ be the Grothendieck group of the category $R(\b)\gmm$, and 
$K_{\gp}(\b)$ the Grothendieck group of the category $R(\b)\gpp$. 
Set
\begin{align*}
\tag{2.3.1}
K_{\gm} = \bigoplus_{\b \in Q_+}K_{\gm}(\b), \qquad
K_{\gp} = \bigoplus_{\b \in Q_+}K_{\gp}(\b).
\end{align*}
By the convolution product, $K_{\gm}$ and $K_{\gp}$ are equipped with 
structures of associative algebras over $\BZ$. 
The shift operation $M \mapsto qM$ on the graded module $M$ 
induces an action of $\BA$ on the Grothendieck groups.  Thus 
$K_{\gm}$ and $K_{\gp}$ have the structure of $\BA$-algebras. 

\para{2.4.}
Let $P$ be a finitely generated projective $R(\b)$-module. We define a dual module
$\BD P$ by 
\begin{equation*}
\tag{2.4.1}
\BD P = \Hom _{R(\b)}(P, R(\b)) = \bigoplus_{n \in \BZ}\Hom_{R(\b)}(q^nP, R(\b))_0,
\end{equation*}
where $\Hom_{R(\b)}(-,-)_0$ is the space of degree preserving homomorphisms.
Hence $\BD P$ is a graded $\Bk$-vector space.
$R(\b)$ acts on $\BD P$ by 
\begin{equation*}
\tag{2.4.2}
x(\la)(m) = \la(\psi(x)m)
\end{equation*}
for $x \in R(\b), \la \in \BD P, m \in P$. 
Then $\BD P \in R(\b)\gpp$. 
\par
Let $M$ be a finite dimensional $R(\b)$-module. Define a dual $\BD M$ by
\begin{equation*}
\tag{2.4.3}
\BD M = \Hom_{\Bk}(M, \Bk) = \bigoplus_{n \in \BZ}\Hom_{\Bk}(q^n M, \Bk)_0.
\end{equation*}
Similarly to (2.4.2), $\BD M$ has a structure of $R(\b)$-module, and 
$\BD M \in R(\b)\gmm$. 
\par
Thus $\BD$ gives a contravariant functor $R(\b)\gpp \to R(\b)\gpp$ or
$R(\b)\gm \to R(\b)\gm$ such that $\BD^2 \simeq \Id$. 
\par
The object $M$ in $R(\b)\gpp$ or in $R(\b)\gmm$ is said to be self-dual
if $\BD M \isom M$. 
\par
For $\b \in Q_+$, let $\bB_{\b}$ be the set of $[P]$ in $K_{\gp}(\b)$, where 
$P$ runs over self-dual finitely generated projective indecomposable $R(\b)$-modules.  
Let $\bB^*_{\b}$ be the set of $[L]$ in $K_{\gm}(\b)$, 
where $L$ runs over self-dual finite dimensional 
simple $R(\b)$-modules. Then $\bB_{\b}$ gives an $\BA$-basis of $K_{\gp}(\b)$, and 
$\bB^*_{\b}$ gives an $\BA$-basis of $K_{\gm}(\b)$. We set 
$\bB = \bigsqcup_{\b \in Q_+}\bB_{\b}$, and $\bB^* = \bigsqcup_{\b \in Q_+}\bB^*_{\b}$. 

\para{2.5.}
For any $P \in R(\b)\gpp, M \in R(\b)\gmm$, we define 
\begin{equation*}
\tag{2.5.1}
\lpp [P], [M] \rpp = \qdim_{\Bk}(P^{\psi}\otimes_{R(\b)}M)  
                   = \sum_{d \in \BZ}\dim_{\Bk}(P^{\psi}\otimes_{R(\b)}M )_d q^d.
\end{equation*}
Then this defines a $\BA$-bilinear map $K_{\gp}(\b) \times K_{\gm}(\b) \to \BA$.
Since $\Bk$ is algebraically closed, an irreducible module $L$ is absolutely 
irreducible.   Then for a finitely generated projective indecomposable module $P$ and 
a finite dimensional simple module $L$, we have
\begin{equation*}
\tag{2.5.2}
\lpp [P], [L]\rpp = \begin{cases}
                      1   &\quad\text{ if $P$ is the projective cover of $\BD L$, } \\
                      0   &\quad\text{ otherwise.}
                  \end{cases}
\end{equation*}
Thus $\lpp\ ,\ \rpp : K_{\gp}(\b) \times K_{\gm}(\b) \to \BA$ gives a perfect pairing on $\BA$, 
and $\bB_{\b}$ and $\bB^*_{\b}$ are dual to each other.   
\par
On the other hand, for any $P, Q \in R(\b)\gpp$, we define 
\begin{equation*}
\tag{2.5.3}
([P], [Q]) = \sum_{d \in \BZ}\dim (P^{\psi}\otimes_{R(\b)}Q)_d q^d.
\end{equation*}
This gives a well-defined symmetric bilinear form 
$(\ ,\ ) : K_{\gp}(\b) \times K_{\gp}(\b) \to \BZ((q))$. 

\para{2.6.}
For $i \in I$, let $L_i = \Bk$ be the (graded) simple $R(\a_i)$-module, and
$P_i$ the projective cover of $L_i$.  Then $P_i = \Bk[x_1]e(i) = R(\a_i)$.   
For any $n \ge 1$, set
\begin{equation*}
\tag{2.6.1}
L_i^{(n)} = q_i^{\binom {n}{2}}L_i^{\circ n}, \qquad 
P_i^{(n)} = q_i^{\binom {n}{2}}P_i^{\,\circ n}, 
\end{equation*}
where $L_i^{\circ n} = \underbrace{L_i\circ \cdots \circ L_i}_{n\text{-times}}$, 
and similarly for $P_i^{\circ n}$. 
Then $L_i^{(n)}$ is a simple $R(n\a_i)$-module lying in $R(n\a_i)\gmm$, 
and $P_i^{(n)} \in R(n\a_i)\gpp$ is the projective cover of $L_i^{(n)}$. 

\par
Assume that $\b, \in Q_+$ and $n \in \BN$.  We further assume that
$\b - n\a_i \in Q_+$.  
We define a functor $E_i^{(n)}: R(\b)\gmm \to R(\b-n\a_i)\gmm$ 
by 
\begin{equation*}
\tag{2.6.2}
M \mapsto \bigl(R(\b-n\a_i)\circ P_i^{(n)}\bigr)^{\psi}\otimes_{R(\b)}M.
\end{equation*}
If $n = 1$, we write $E_i^{(n)}$ as $E_i$.  $E_i^{(n)}M$ is also written as
\begin{align*}
E_i^{(n)}M  &= \bigl(R(\b - n\a_i)\otimes_{\Bk}(P_i^{(n)})^{\psi}\bigr)
                    \otimes_{R(\b-n\a_i)\otimes R(n\a_i)}M  \\
            &= \bigl(R(\b - n\a_i)\otimes_{\Bk}(P_i^{(n)})^{\psi}\bigr)
                    \otimes_{R(\b-n\a_i)\otimes R(n\a_i)}e(\b-n\a_i, n\a_i)M.
\end{align*}
The action of $R(\b-n\a_i)\otimes R(n\a_i)$ on $M$ 
is equal to the $R(\b-n\a_i)\otimes R(n\a_i)$-module $e(\b- n\a_i, n\a_i)M$, and 
it acts as 0 outside. 
Since $P_i = R(\a_i)$, we have
$R(\b-\a_i)\circ P_i = R(\b)e(\b-\a_i,\a_i)$.  Hence 

\begin{equation*}
\tag{2.6.3}
E_iM = e(\b-\a_i, \a_i)R(\b)\otimes_{R(\b)}M = e(\b-\a_i,\a_i)M.
\end{equation*}

\par
$E_i^{(n)}$ is an exact functor, and it induces an $\BA$-homomorphism 
\begin{equation*}
{e'}_i^{(n)} : K_{\gm}(\b) \to K_{\gm}(\b- n\a_i). 
\end{equation*}

\par
For $\b \in Q_+, n \in \BN$, define an 
additive functor 
$F_i^{(n)}: R(\b)\gpp \to R(\b + n\a_i)\gpp$ by 
\begin{equation*}
\tag{2.6.4}
P \mapsto P \circ P_i^{(n)}.
\end{equation*}
The functor $F_i^{(n)}$ induces an $\BA$-homomorphism on the Grothendieck groups, 
\begin{equation*}
f_i^{(n)} : K_{\gp}(\b) \to K_{\gp}(\b + n\a_i).
\end{equation*}

It is known that the operators $f_i^{(n)}$ on $K_{\gp}$, and 
$e_i'^{(n)}$ on $K_{\gm}$ 
satisfy the following 
adjunction relation with respect to the pairing $\lpp\ ,\ \rpp$.
(The proof is also given as the special case of Lemma 3.10). 

\begin{equation*}
\tag{2.6.5}
\lpp f_i^{(n)}[P], [M]\rpp = \lpp [P], e'^{(n)}_i[M]\rpp.
\end{equation*} 

The following categorification theorem was proved by 
Khovanov-Lauda [KL], and Rouquier [R1]. 

\begin{thm}[{[KL], [R1]}]  
There exists an isomorphism 
of $\BA$-algebras
\begin{equation*}
\tag{2.7.1}
\wt\vT_0 : {}_{\BA}\BU_q^- \isom K_{\gp}, 
\end{equation*}
which maps $f_i^{(n)}$ to $[P_i^{(n)}]$.  
Moreover, $\wt\vT_0$ is an isometry. 
\end{thm}

\para{2.8.}
Let $[\Bk]$ be the isomorphism class of the trivial representation 
$\Bk$ of $R(0)$, which gives a basis of $K_{\gm}(0)$ and of $K_{\gp}(0)$. 
Let $1^* \in {}_{\BA}(\BU_q^-)^*_0$ be the dual basis of $1 \in {}_{\BA}(\BU_q^-)_0$. 
As a corollary to Theorem 2.7, we have the following result.

\begin{prop}   
\begin{enumerate}
\item \ 
There exists an isomorphism of $\BA$-modules
 (actually, an anti-algebra isomorphism), 
\begin{equation*}
\wt\vT  : {}_{\BA}\BU_q^- \isom K_{\gp}, 
\end{equation*}
which maps $1 \in {}_{\BA}(\BU_q^-)_0$ to $[\Bk] \in K_{\gp}(0)$.  
$\wt \vT$ commutes with the actions of $f_i^{(n)}$. 
$\wt\vT$ is an isometry with respect to the inner product $(\ , \ )$
on ${}_{\BA}\BU_q^-$ and on $K_{\gp}$. 
\item \
There exists an isomorphism of $\BA$-modules, 
\begin{equation*}
\wt \vT^* : K_{\gm} \isom {}_{\BA}(\BU_q^-)^* 
\end{equation*}
which maps $[\Bk]$ to $1^*$.  Moreover, $\wt \vT^*$ commutes with 
the actions of $e'^{(n)}_i$. 
\end{enumerate}
\end{prop}

\begin{proof}
We define $\wt\vT : {}_{\BA}\BU_q^- \to K_{\gp}$ by 
$\wt\vT = \wt\vT_0 \circ *$, where $* : {}_{\BA}\BU_q \to {}_{\BA}\BU_q^-$
is the anti-involution on ${}_{\BA}\BU_q^-$ (see 1.2). 
Since the action of $f_i^{(n)}$ on $K_{\gp}$ is defined by the functor 
$F_i^{(n)}: R(\b)\gpp \to R(\b + n\a_i)\gpp, P \mapsto P \circ P_i^{(n)}$, 
it is clear from (2.7.1) that $\wt\vT$ commutes with $f_i^{(n)}$. 
Since the inner product on $\BU_q^-$ is invariant under the action of $*$, 
$\wt\vT$ is an isometry.  Hence (i) holds. 
By using $\lpp\ ,\ \rpp$, we have an isomorphism 
$K_{\gm}(\b) \simeq \Hom_{\BA}(K_{\gp}(\b), \BA)$. Thus one can define 
a map $\wt\vT^*$ as the transpose of $\wt\vT$. 
Then by (2.6.5), 
the map $\wt\vT^*$ commutes with $e_i'^{(n)}$.  Hence (ii) holds. 
\end{proof}

\para{2.10.}
From now on , we consider the diagram automorphism 
$\s: I \to I$. $\s$ acts on $\nu = (\nu_1, \dots, \nu_n) \in I^n$
by $\s\nu = (\s(\nu_1), \dots, \s(\nu_n))$. In view of the condition 
(iv) in 2.1, the assignment 
$e(\nu) \mapsto e(\s(\nu)), x_i \mapsto x_i, \tau_i \mapsto \tau_i$ 
induces an algebra isomorphism $R(\b) \isom R(\s(\b))$.
Thus $\s$ gives a functor $\s^* : R(\s(\b))\Mod \to R(\b)\Mod$, 
where, for an $R(\s(\b))$-module $M$, $\s^*M$ is the pull-back of $M$ by 
the homomorphism  $R(\b) \to R(\s(\b))$.  
We consider the functor $\tau = (\s\iv)^* : R(\b)\Mod \to R(\s(\b))\Mod$ 
for any $\b \in Q_+$. 
Then $\tau$ gives functors $R(\b)\gmm \to R(\s(\b))\gmm$, and 
$R(\b)\gpp \to R(\s(\b))\gpp$. 
Thus $\tau$ induces automorphisms on 
$K_{\gp}$ and on $K_{\gm}$, which we also denote by $\tau$. 
Note that $\tau$ permutes the bases $\bB$ and $\bB^*$. 
\par
Since $\tau(L_i^{(n)}) = L_{\s(i)}^{(n)}, \tau(P_i^{(n)}) = P_{\s(i)}^{(n)}$, 
we see that
\begin{equation*}
\tag{2.10.1}
\tau\circ f_i^{(n)} = f_{\s(i)}^{(n)} \circ \tau  \quad\text{ on } \ K_{\gp}, \qquad 
\tau\circ e'^{(n)}_i = e'^{(n)}_{\s(i)}\circ \tau \quad\text{ on } \ K_{\gm}.
\end{equation*}

It follows that the map $\wt\vT : {}_{\BA}\BU_q^- \isom K_{\gp}$ is 
compatible with the actions of $\s$ and $\tau$.     
Hence the map $\wt\vT^* : K_{\gm} \isom {}_{\BA}(\BU_q^-)^*$ is also
compatible with the actions of $\tau$ and $\s$. 

\para{2.11.}
We follow the setup in 1.6.  Hence we assume that $\Bn$ is a power of 
a prime number $\ell$, and set 
$\BA' = \BF[q,q\iv]$. 
We define ${}_{\BA'}K_{\gp} = \BA'\otimes_{\BA}K_{\gp}$, 
${}_{\BA'}K_{\gm} = \BA'\otimes_{\BA}K_{\gm}$.  Then $\bB$ and $\bB^*$ give 
the $\BA'$-bases of ${}_{\BA'}K_{\gp}$, and ${}_{\BA'}K_{\gm}$, respectively. 
The pairing $\lpp\ ,\ \rpp$ is extended  to a perfect $\BA'$-pairing 
${}_{\BA'}K_{\gp} \times {}_{\BA'}K_{\gm} \to \BA'$.  
\par
Let ${}_{\BA'}K_{\gp}^{\tau}$ be the set of $\tau$-fixed elements in 
${}_{\BA'}K_{\gp}$.
Then ${}_{\BA'}K_{\gp}^{\tau}$ is an $\BA'$-submodule of ${}_{\BA'}K_{\gp}$, and  
$\{ O(b) \mid b \in \bB\}$ gives an $\BA'$-basis of $_{\BA'}K_{\gp}^{\tau}$.
Let $\CJ$ be the subset of ${}_{\BA'}K_{\gp}^{\tau}$ consisting of $O(x)$ for 
$x \in {}_{\BA'}K_{\gp}$ such that $O(x) \ne x$. Then $\CJ$ is an $\BA'$-submodule 
of ${}_{\BA'}K_{\gp}^{\tau}$.  
We define the quotient module $\CV_q$ by
\begin{equation*}
\tag{2.11.1}
\CV_q = {}_{\BA'}K_{\gp}^{\tau}/\CJ.
\end{equation*}

Since the isomorphism $\wt\vT : {}_{\BA}\BU_q^- \isom K_{\gp}$ 
is compatible with the actions of $\s$ and $\tau$ by 2.10, 
it induces an isomorphism ${}_{\BA'}\BU_q^{-,\s} \isom {}_{\BA'}K^{\tau}_{\gp}$, 
which maps $\BJ$ onto $\CJ$.  Thus
$\wt\vT$ induces an isomorphism 
\begin{equation*}
\tag{2.11.2}
\vT : \BV_q \isom \CV_q.
\end{equation*}

\par
For each $j \in J$, let $j = \{ i_1, \dots, i_t\}$. 
We define a functor $\wt F^{(n)}_j : R\gpp \to R\gpp$ 
by  $\wt F^{(n)}_j = F_{i_t}^{(n)}\cdots F_{i_1}^{(n)}$.  
This induces an $\BA'$-homomorphism 
$\wt f_j^{(n)} = f_{i_t}^{(n)}\cdots f_{i_1}^{(n)}$ on ${}_{\BA'}K_{\gp}$.
Since $f^{(n)}_{i_k}$ commute each other by Theorem 2.7, and $\tau$ permutes $f_{i_k}^{(n)}$, 
$\wt f_j^{(n)}$ commutes with $\tau$, and acts on $\CV_q$.   
The map $\vT : \BV_q \to \CV_q$ is compatible 
with the action of $\wt f_j^{(n)}$.
\par
Next consider the $\BA'$-submodule ${}_{\BA'}K_{\gm}^{\tau}$ consisting of $\tau$-fixed
elements. We define $\CJ^*$ as the $\BA'$-submodule of ${}_{\BA'}K_{\gm}^{\tau}$ 
consisting of $O(x)$ for $x \in {}_{\BA'}K_{\gm}$ such that $O(x) \ne x$. 
We define a quotient module $\CV_q^*$ by 
\begin{equation*}
\tag{2.11.3}
\CV_q^* = {}_{\BA'}K_{\gm}^{\tau}/\CJ^*.
\end{equation*}

\par
Similarly to the case of $\wt\vT$, 
the isomorphism $\wt\vT^* : K_{\gm} \isom {}_{\BA}(\wt\BU_q^-)^*$ 
induces an isomorphism ${}_{\BA'}K_{\gm}^{\tau} \isom {}_{\BA'}(\wt\BU_q^-)^{*,\s}$, 
which maps $\CJ^*$ onto $\BJ^*$. 
Thus $\wt\vT^*$ induces an isomorphism 
\begin{equation*}
\tag{2.11.4}
\vT^* : \CV_q^* \isom \BV_q^*.
\end{equation*}   

\par
For $j \in J$, we define a functor $\wt E^{(n)}_j : R\gmm \to R\gmm$ 
by $\wt E_j^{(n)} = E_{i_1}^{(n)}\cdots E_{i_t}^{(n)}$.  This induces
an $\BA'$-homomorphism $\wt e'^{(n)}_j = {e'}_{i_1}^{(n)}\cdots {e'}_{i_t}^{(n)}$, 
commuting with $\tau$.  Thus, it induces an action of $\wt e'^{(n)}_j$ 
on $\CV_q^*$.
The map  $\vT^*$ commutes with the action 
of $\wt{e'}_j^{(n)}$.
\par
Let $\ul{\bB}_{\bullet}$ be the image of $\bB^{\tau}$ onto $\CV_q$, 
and $\ul{\bB}^*_{\bullet}$ the image 
of $\bB^{*,\tau}$ onto $\CV_q^*$. 
Then $\ul\bB_{\bullet}$ gives a basis of $\CV_q$, and 
$\ul{\bB}^*_{\bullet}$ gives a basis of $\CV_q^*$, which are dual to each other. 
\par
Summing up the above discussion,  we obtain the following.

\begin{prop}  
There exist isomorphisms of $\BA'$-modules, 
\begin{equation*}
\vT : \BV_q \isom \CV_q, \qquad \vT^* : \CV_q^* \isom \BV_q^*.
\end{equation*}
\par
The map $\vT$ commutes with the actions of $\wt f_j^{(n)}$, 
and the map $\vT^*$ commutes with the actions of $\wt{e'}_j^{(n)}$.  
\end{prop}

\par\bigskip
\section{Periodic functors associated to KLR algebras }

\para{3.1.}
First we recall the general theory of periodic functors due to 
[L2, 11.1].  
Let $\ZC$ be a linear category over a field $\Bk$, 
namely, a category where the space of morphisms between 
any two objects has the structure of $\Bk$-vector space such that the composition 
of morphisms is bilinear, and that finite direct sums exist.  A functor between 
two linear categories is said ro be linear if it preserves the $\Bk$-vector space structure.
\par
A linear functor $\s^* : \ZC \to \ZC$ is said to be periodic if there exists 
$\Bn \ge 1$ such that $(\s^*)^{\Bn} = \id_{\ZC}$. 
Assume that $\s^*$ is a periodic functor on $\ZC$.  We define a new category 
$\wt\ZC$ as follows. The objects of $\wt\ZC$ are the pairs $(A, \f)$, where 
$A \in \ZC$, and $\f : \s^*A \isom A$ is an isomorphism in $\ZC$ such that the composition
satisfies the relation 
\begin{equation*}
\tag{3.1.1}
\f\circ \s^*\f \circ\cdots\circ (\s^*)^{\Bn-1}\f = \id_A.
\end{equation*}
A morphism from $(A,\f)$ to $(A',\f')$  in $\wt\ZC$ is a morphism 
$f : A \to A'$ in $\ZC$ satisfying the following commutative diagram, 
\begin{equation*}
\tag{3.1.2}
\xymatrix@C=36pt@ R=30pt@ M =8pt{
\s^*A     \ar[r]^-{\s^*f}  \ar[d]^{\f} 
       &  \s^*A ' \ar[d]^-{\f'}    \\  
A  \ar[r]^-{f}  &  A'. 
}
\end{equation*}

The category $\wt\ZC$ is a linear category, hence is an additive category.  
$\wt\ZC$ is called 
the category associated to the periodic functor $\s^*$ on $\ZC$ 

\para{3.2.}
Let $\wt\ZC$ be as in 3.1. 
Let $\Bn$ be the smallest integer such that $(\s^*)^{\Bn} = \id_{\ZC}$.
Assume that $\Bk$ is an algebraically closed field such that
$\ch \Bk$ does not divide $\Bn$. 
Let $\z_{\Bn}$ be a primitive $\Bn$-th root of unity in $\BC$, and fix, once 
for all, a ring homomorphism $\BZ[\z_{\Bn}] \to \Bk$, which maps $\z_{\Bn}$ 
to a primitive $\Bn$-th root of unity in $\Bk$.  
If $(A, \f) \in \wt\ZC$, then $(A, \z_{\Bn}\f) \in \wt\ZC$, which we denote by
$\z_{\Bn}(A,\f)$. 
\par
An object $(A, \f) \in \wt\ZC$ is said to be traceless 
if there exists an object $B \in \ZC$, an integer $t\ge 2$ which is a divisor of $\Bn$ such that 
$(\s^*)^t B \simeq B$, and an isomorphism 
\begin{equation*}
A \simeq B \oplus \s^*B \oplus \cdots \oplus (\s^*)^{t-1}B,
\end{equation*}
where, $\f: \s^*A \isom A$ is given 
by the identity maps $(\s^*)^kB \isom (\s^*)^kB$ ($1 \le k \le t-1$) and 
an isomorphism $(\s^*)^tB \isom B$ on each direct summand, 
under the above isomorphism. 
\par
For the category $\wt\ZC$,  
$K(\wt\ZC)$ is defined as a $\BZ[\z_{\Bn}]$-module generated by 
symbols $[A, \f]$ associated  to the isomorphism class of the object $(A,\f) \in \wt\ZC$,
subject to the relations 
\par\medskip
\begin{enumerate}
\item \ $[X] = [X'] + [X'']$  if $X \simeq X'\oplus X''$, 
\item \ $[A, \z_{\Bn}\f] = \z_{\Bn}[A,\f]$, 
\item \ $[X] = 0$ if $X$ is traceless.
\end{enumerate}
In the case where $\wt\ZC$ is an abelian category, 
the condition (i) is replaced by the condition 
\par\medskip
(i$'$) \ $[X] = [X'] + [X'']$ if there exists a short exact sequence
\begin{equation*}
0 \to X' \to X \to X'' \to 0.
\end{equation*}

$K(\wt\ZC)$ is an analogue of the Grothendieck group.  In fact, if 
$\Bn = 1$, $\wt\ZC = \ZC$, then $K(\wt\ZC)$ coincides with the Grothendieck group 
$K(\ZC)$. We call $K(\wt\ZC)$ the Grothendieck group of $\wt\ZC$.  
\par\medskip

\para{3.3.}
We consider the KLR algebra $R = \bigoplus_{\b \in Q_+}R(\b)$ 
associated to $(X, \s)$ as in 2.1. Let $\Bn$ be the order of $\s$, 
and assume that $\ch \Bk$ does not divide $\Bn$ (here we don't give any 
restriction on $\Bn$).  
Assume that $\b$ is $\s$-stable.  Then by 2.10, we obtain a functor 
$\s^*: R(\b)\Mod \to R(\b)\Mod$. 
We apply the discussion in 3.1, and let $\SC_{\b}$ be 
the category $\wt\ZC$ associated to the periodic functor $\s^*$ on 
$\ZC = R(\b)\Mod$.  
Then as remarked in [M, Lemma 2.2], the category $\SC_{\b}$ is equivalent 
to the category of graded representations of the algebra
\begin{equation*}
\tag{3.3.1}
R(\b)\sharp (\BZ/\Bn\BZ) = R(\b) \otimes_{\Bk}\Bk[\BZ/\Bn \BZ],  
\end{equation*}
where $\Bk[\BZ/\Bn\BZ]$ is the group algebra of $\BZ/\Bn\BZ$, and 
$\Bk[\BZ/\Bn\BZ]$ acts on $R(\b)$ through the action of $\s^*$.
Hence $\SC_{\b}$ is an abelian category.  
We denote by $\SP_{\b}$ the full subcategory of $\SC_{\b}$ consisting 
of finitely generated projective objects in $\SC_{\b}$.  
Let $\SL_{\b}$ be the full subcategory of $\SC_{\b}$
consisting of $(M,\f)$ such that $M $ is a finite dimensional $R(\b)$-module. 
Then $\SP_{\b}$ is an additive category and $\SL_{\b}$ is an abelian category. 
\par
Let $K(\SP_{\b})$ (resp. $K(\SL_{\b})$) be the Grothendieck group of 
the category $\SP_{\b}$ (resp. $\SL_{\b}$) as defined in 3.2.  
$K(\SP_{\b})$ and $K(\SL_{\b})$ have structures of $\BZ[\z_{\Bn}, q, q\iv]$-modules,  
where $q$ acts as the grading shift as in the case of $K_{\gp}(\b)$ or $K_{\gm}(\b)$. 
We set
\begin{equation*}
K(\SP) = \bigoplus_{\b \in Q_+^{\s}}K(\SP_{\b}), \qquad
K(\SL) = \bigoplus_{\b \in Q_+^{\s}}K(\SL_{\b}). 
\end{equation*} 

\para{3.4.}
The contravariant functor $\BD$ on $R(\b)\gmm$ or on $R(\b)\gpp$ given 
in 2.4 can be extended to the case of $\SL_{\b}$ or $\SP_{\b}$ as follows.  
For $(M,\f)$ in $\SP_{\b}$ or in $\SL_{\b}$, we define the dual object by
\begin{equation*}
\BD (M,\f) = (\BD M, \BD(\f)\iv),
\end{equation*}
where $\BD(\f) : \BD M \isom \BD(\s^*M) = \s^*\BD M$. 
It is clear that $\BD$ gives a functor $\SL_{\b} \to \SL_{\b}$.  
It was shown in [M, 5] that $\BD$ gives a functor $\SP_{\b} \to \SP_{\b}$. 
\par
$(M,\f)$ in $\SL_{\b}$ or in $\SP_{\b}$ is said to be self-dual if 
$\BD (M,\f) \simeq (M,\f)$ in that category. 
Note that if $(M,\f)$ is a self-dual object, then $\f$ is uniquely 
determined by $M$ in the case where $\Bn$ is odd, and unique up to
$\pm 1$ in the case where $\Bn$ is even. 
In particular, $(\Bk, \id) \in \SL_0$ is a unique self-dual simple object
if $\Bn$ is odd, and there exist two self-dual simple object 
$(\Bk, \pm \id) \in \SL_0$. We write $[\pm 1] = [\Bk, \pm\id ] \in K(\SL_0)$.  
\par
Let $\wt{\ul\bB}^*_{\b}$ be the set of elements in $K(\SL_{\b})$ 
consisting of isomorphism classes of self-dual objects $(M,\f)$ such that
$M$ is a finite dimensional simple $R(\b)$-module.
It was proved in [M, Thm.10.8], by using the crystal structure of 
$\wt{\ul\bB}^* = \bigsqcup_{\b}\wt{\ul\bB}^*_{\b}$, 
 that there exists 
a $\BZ[\z_{\Bn}, q, q\iv]$-basis $\ul\bB^*_{\b}$ of $K(\SL_{\b})$ 
such that $\ul\bB^*_{\b} \subset \wt{\ul\bB}^*_{\b}$ 
and that $[1] \in \ul\bB^*_0$. 
Also there exists a $\BZ[\z_{\Bn}, q, q\iv]$-basis 
$\ul\bB_{\b}$ of $K(\SP_{\b})$ consisting of isomorphism classes of  
finitely generated projective indecomposable self-dual objects 
in $\SP_{\b}$ which 
are projective cover of elements in $\ul\bB_{\b}^*$. 
We set $\ul\bB = \bigsqcup_{\b \in Q_+^{\s}}\ul\bB_{\b}$, and 
$\ul\bB^* = \bigsqcup_{\b \in Q_+^{\s}}\ul\bB_{\b}^*$. 

\para{3.5.}
We extend the pairing $\lpp\ ,\ \rpp : K_{\gp}(\b) \times K_{\gm}(\b) \to \BA$ 
given in 2.5 to our situation.
We define a bilinear map 
$\lpp\ ,\ \rpp : K(\SP_{\b}) \times K(\SL_{\b})  \to \BZ[\z_{\Bn}, q, q\iv]$ by 

\begin{equation*}
\tag{3.5.1}
\lpp [P, \f], [M, \f']\rpp 
   = \sum_{d \in \BZ}\Tr\bigl(\f\otimes \f', (P^{\psi}\otimes_{R(\b)} M)_d\bigr) q^d.  
\end{equation*}
Here for given $\f : \s^*P \isom P, \f': \s^*M \isom M$, 
$\f\otimes \f'$ is the induced map $\s^*P^{\psi}\otimes \s^*M \to P^{\psi}\otimes M$. 
Since $\s^*P^{\psi}\otimes \s^*M$ is canonically isomorphic to 
$P^{\psi}\otimes M$, 
$\f\otimes \f'$ is regarded as a linear transformation on $P^{\psi}\otimes M$.  
\par
Note that in [M, Lemma 3.1], instead of $\lpp\ ,\ \rpp$, 
 a sesqui-linear form $\lp\ ,\ \rp$ is used, which is defined by 
\begin{equation*}
\lp\, [P, \f], [M, \f']\rp = \sum_{d \in \BZ}
               \Tr(\s_{\f,\f'}, \Hom_{R(\b)}(P, M)_d)q^d.
\end{equation*}
The relation between these forms is given 
by the isomorphism of graded vector spaces,
\begin{equation*}
\tag{3.5.2}
\Hom_{R(\b)}(P, M) \simeq \BD P^{\psi}\otimes_{R(\b)}M.
\end{equation*}
\par
By [M, Thm.11.1], we have the following.
\par\medskip\noindent
(3.5.3) \ The pairing 
$\lpp\ ,\ \rpp : K(\SP_{\b}) \times K(\SL_{\b}) \to \BZ[\z_{\Bn}, q,q\iv]$
gives a perfect pairing.  $\ul\bB$ and $\ul\bB^*$ are dual to each other.  
\par\medskip
The symmetric bilinear form $(\ ,\ ) : K_{\gp}(\b) \times K_{\gp}(\b) \to \BZ((q))$
is also extended to our situation as follows.
For $(P,\f), (Q,\f') \in \SP_{\b}$, set
\begin{equation*}
\tag{3.5.4}
([P,\f], [Q,\f']) = \sum_{d \in \BZ}\Tr(\f\otimes \f', (P^{\psi}\otimes_{R(\b)}Q)_d)q^d.
\end{equation*}
Here $\f\otimes \f'$ gives an isomorphism 
$\s^*P^{\psi} \otimes_{R(\b)} \s^*Q \isom P^{\psi}\otimes_{R(\b)}Q$.  By using the canonical 
isomorphism $\s^*P^{\psi}\otimes_{R(\b)}\s^*Q \simeq P^{\psi}\otimes_{R(\b)}Q$, we regard 
$\f\otimes \f'$ as a linear transformation on $P^{\psi}\otimes_{R(\b)}Q$. 
(3.5.3) induces a symmetric bilinear pairing 
$(\ ,\ ): K(\SP_{\b}) \times K(\SP_{\b}) \to \BZ[\z_{\Bn}]((q))$. 
\para{3.6.}
We extend the convolution product defined in 2.2 to the case of 
$\SC_{\b}$.  Take $(M, \f) \in \SC_{\b}, (N,\f') \in \SC_{\g}$.
Set 
\begin{equation*}
\tag{3.6.1}
(M,\f)\circ (N,\f') = (M\circ N, \f \circ \f'),
\end{equation*}
where $\f\circ \f'$ is given by the composite of the map 
$\s^*M \circ \s^*N  \to M \circ N$ and the canonical isomorphism 
$\s^*(M\circ N) \simeq \s^*M \circ \s^*N$. 
\par
For $\b, \g \in Q_+^{\s}$, the automorphisms $\s$ on $R(\b)$ and $R(\g)$ 
induces an automorphism $\s$ on $R(\b)\otimes R(\g)$ by 
$\s(u\otimes w) = \s(u)\otimes \s(w)$.  Thus one can consider 
the category of graded representations of 
$(R(\b)\otimes R(\g))\sharp(\BZ/\Bn\BZ)$, 
which we denote by $\SC_{\b \sqcup \g}$. 
We can define full subcategories $\SL_{\b \sqcup \g}, \SP_{\b \sqcup \g}$ of 
$\SC_{\b \sqcup \g}$ similarly. 
\par
The convolution product (3.6.1) gives an induction functor 
$\Ind_{\b,\g} : \SC_{\b \sqcup \g} \to \SC_{\b + \g}$.
On the other hand, for each $(M, \f) \in \SC_{\b+\g}$, we define 
\begin{equation*}
\tag{3.6.2}
\Res_{\b,\g}(M,\f) = (e(\b,\g)M, \f'),
\end{equation*}
where $\Res M = e(\b,\g)M$ is 
an $R(\b)\otimes R(\g)$-module given in 2.2.
We have the canonical isomoprhism 
$\s^*(\Res M) \isom \Res(\s^*M)$ since $e(\b,\g)$ is $\s$-invariant, 
which induces the isomorphism  $\f': \s^*(\Res M) \isom \Res M$. 
Thus $\Res_{\b,\g}(M,\f)$ belongs to $\SC_{\b \sqcup \g}$, and   
we obtain the restriction functor 
$\Res_{\b,\g} : \SC_{\b + \g} \to \SC_{\b\sqcup \g}$. 
\par
It is known by [M, Thm. 4.3] that the induction functor and the restriction 
functor form an adjoint pair of exact functors between $\SC_{\b \sqcup \g}$ 
and $\SC_{\b+\g}$. 
\par
We have isomorphisms of $\BZ[\z_{\Bn}, q,q\iv]$-modules
\begin{equation*}
K(\SP_{\b})\otimes_{\BZ[\z_{\Bn}, q^{\pm 1}]}K(\SP_{\g}) \simeq K(\SP_{\b\sqcup \g}), 
\quad
K(\SL_{\b})\otimes_{\BZ[\z_{\Bn}, q^{\pm 1}]}K(\SL_{\g}) \simeq K(\SL_{\b\sqcup \g}), 
\end{equation*} 
by [M, Prop. 4.1], and the induction functors $\Ind_{\b,\g}$ induces homomorphisms 
of Grothendieck groups,
\begin{equation*}
K(\SP_{\b}) \otimes_{\BZ[\z_{\Bn}, q^{\pm 1}]}K(\SP_{\g}) \to K(\SP_{\b + \g}), \quad 
K(\SL_{\b}) \otimes_{\BZ[\z_{\Bn}, q^{\pm 1}]} K(\SL_{\g}) \to K(\SL_{\b + \g}). 
\end{equation*}
It is shown that $K(\SP)$ and $K(\SL)$ turn out to be associative algebras 
over $\BZ[\z_{\Bn}, q^{\pm 1}]$ with respect to this product.  

\para{3.7.}
Let $j = \{ i_1, \dots, i_t \} \in J$, and set  
$\a_j = \a_{i_1} + \cdots + \a_{i_t}$, then $\a_j \in Q_+^{\s}$. 
We consider $R(\a_j)$. 
Let $\nu = (i_1, \dots, i_t)$.
Then $I^{\a_j} = \{ w\nu \mid w \in S_t\}$, and in particular, for any 
$\nu' = (\nu_1', \dots, \nu_t') \in I^{\a_j}$, $\nu_1', \dots, \nu_t'$ 
are all distinct.  
Moreover, $Q_{i,i'}(u,v) \in \Bk^*$ since $(\a_i, \a_{i'}) = 0$ if $i \ne i' \in j$.
Hence by the defining relations on the KLR algebra $R(\a_j)$ im 2.1,  
we have, for $1 \le k < t$, 

\begin{align*}
\tag{3.7.1}
\tau_k^2 e(\nu) \in \Bk^*e(\nu),  
   \qquad  \tau_kx_{\ell} = x_{s_k(\ell)}\tau_k,
   \qquad  \tau_k\tau_{k+1}\tau_k = \tau_{k+1}\tau_k\tau_{k+1}.
\end{align*}

In particular, $\tau_w$ is well-defined for $w \in S_t$. 
Let 
\begin{equation*}
\BP[x_1, \dots, x_t] = \bigoplus_{\nu' \in I^{\a_j}}\Bk[x_1, \dots, x_t]e(\nu')
                     = \bigoplus_{w \in S_t}\Bk[x_1, \dots, x_t]e(w\nu)
\end{equation*}
be the polynomial ring of $R(\a_j)$. 
By the basis theorem, $R(\a_j)$ is a free $\BP[x_1, \dots, x_t]$-module with 
basis $\{ \tau_w \mid w \in S_t\}$, namely, 
\begin{equation*}
\tag{3.7.2}
R(\a_j) = \bigoplus_{w \in S_t}\BP[x_1, \dots, x_t]\tau_w.
\end{equation*}

For any $f \in \Bk[x_1, \dots, x_t]$, we have 
$\tau_w f = r_w(f) \tau_w$, where $r_w \in \Bk[S_t]$ is the element 
corresponding to $w \in S_t$. Thus
$\tau_w\BP[x_1, \dots, x_t] = \BP[x_1, \dots, x_t]\tau_w$.
The left multiplication of $\tau_w$ sends $\BP[x_1, \dots, x_t]\tau_{w'}$ to 
$\BP[x_1, \dots, x_t]\tau_{ww'}$. 

\par
Recall that 
$R(\a_{i_k}) = P_{i_k} = \Bk[x_{i_k}]e(i_k)$ as $R(\a_{i_k})$-modules.
Then 
\begin{align*}
P_{i_1}\circ \cdots \circ P_{i_t} 
       &= R(\a_j)\otimes_{R(\a_{i_1})\otimes \cdots \otimes R(\a_{i_t})}
                 P_{i_1}\otimes \cdots \otimes P_{i_t}  \\
       &\simeq \bigoplus_{w \in S_t}\tau_w\Bk[x_1, \dots, x_t]e(i_1, \dots, i_t)  \\
       &= \bigoplus_{w \in S_t}\Bk[x_1, \dots, x_t]\tau_w e(i_1, \dots, i_t).
\end{align*}

Set $P_j = P_{i_1}\circ \cdots \circ P_{i_t}$.  Then $P_j$ is written as 
$P_j = \bigoplus_{w \in S_t}\Bk[x_1, \dots, x_t][w]$, where 
$[w] = \tau_{w}e(i_1, \dots, i_t)$.  
Thus $P_j$ is a free $\Bk[x_1, \dots x_t]$-module with 
basis $[w]$, where $e(\nu')$ acts on $[w]$ by 1 if $\nu' = w\nu$, and  
by 0 otherwise, $\tau_i$ acts on $[w]$ by $[s_i(w)]$.
\par
We define a quotient module $L_j$ of $P_j$ by 
\begin{equation*}
L_j = P_j/\lp x_kP_j \mid  1 \le k \le t\rp.
\end{equation*} 
Then $L_j \simeq \bigoplus_{w \in S_t}\Bk [w] \simeq L_{i_1}\circ \cdots \circ L_{i_t}$, 
where $x_k$ acts on $[w]$ 
as 0.  Thus $P_j$ and $L_j$ are isomorphic to the ones defined 
in [M, 7]. 
$L_j$ is a simple $R(\a_j)$-module lying in $R(\a_j)\gmm$, and 
$P_j$ is the projective cover of $L_j$ lying in $R(\a_j)\gpp$. 

\par
We have $P_j = R(\a_j)e(\nu)$ with $\nu = (i_1, \dots, i_t)$, and 
\begin{equation*}
\tag{3.7.3}
R(\a_j) = \bigoplus_{\nu' \in I^{\a_j}}R(\a_j)e(\nu') 
        = \bigoplus_{\nu' \in I^{\a_j}}P_{j, \nu'},
\end{equation*}
where $P_{j,\nu'} = \bigoplus_{w \in S_t}\Bk[x_1, \dots, x_t]\tau_we(\nu')$, 
which is isomorphic to $P_j$ as $R(\a_j)$-modules. 

\par
Since $\a_j$ is $\s$-stable, one can consider the category 
$\SC_{\a_j}, \SL_{\a_j}$ and $\SP_{\a_j}$. 
We define $\f_j : \s^*L_j \isom  L_j$ by $\f_j[w] = [\s(w)]$ (here 
$[w]$ is identified with the permutation $(i_1', \dots, i_t')$ of 
$(i_1, \dots, i_t)$ on which $\s$ acts).  We define 
$L(j) = (L_j, \f_j)$, which is a simple object in $\SL_{\a_j}$.  
We define $P(j)$ the projective cover of $L(j)$, which lies in $\SP_{\a_j}$.
$P(j)$ is given as $P(j) = (P_j, \f_j)$, where $\f_j : \s^*P_j \isom P_j$ 
is the extension of $\f_j : \s^*L_j \isom  L_j$.  

\para{3.8.}
For $n \ge 1$, we define $L(j)^{(n)} \in \SL_{n\a_j}$ by 
\begin{equation*}
\tag{3.8.1}
L(j)^{(n)} = q_j^{\binom{n}{2}}L(j)^{\circ n}.
\end{equation*}
By [M, Lemma 7.2], $L(j)^{(n)}$ is a self-dual simple object in $\SL_{n\a_j}$. 
Note that $L(j)^{(n)} = (L_j^{(n)}, \f_{nj})$, 
where $L_j^{(n)} = q_j^{\binom {n}{2}}L_j^{\circ n}$, and 
$\f_{nj} : \s^*L_j^{(n)} \isom L_j^{(n)}$ is induced from $\f_j : \s^*L_j \isom L_j$.  
Let $P(j)^{(n)}$ be the projective cover of $L(j)^{(n)}$.  
Then 
\begin{equation*}
\tag{3.8.2}
P(j)^{(n)} = q_j^{\binom {n}{2}} P(j)^{\circ n}.
\end{equation*}
We have $P(j)^{(n)} = (P_j^{(n)}, \f_{nj})$, where 
$P_j^{(n)} = q_j^{\binom {n}{2}}P_j^{\circ n}$ is the projective cover of $L_j^{(n)}$, 
and $\f_{nj} : \s^*P_j^{(n)} \isom P_j^{(n)}$ is the extension of 
$\f_{nj} : \s^*L_j^{(n)} \isom L_j^{(n)}$. 

\para{3.9.}
Assume that $\b \in Q_+^{\s}$ and $j \in J$. 
We define a functor $F_j^{(n)} : \SC_{\b} \to \SC_{\b + n\a_j}$ by
\begin{equation*}
\tag{3.9.1}
F_j^{(n)} : P \mapsto P \circ P(j)^{(n)}.
\end{equation*}
If we write $P = (P_0, \f)$, then the functor is given by 
\begin{equation*}
F_j^{(n)}P 
    = \bigl(R(\b + n\a_j)\otimes_{R(\b)\otimes R(n\a_j)}
         (P_0 \otimes P_j^{(n)}), \f'\bigr),
\end{equation*}
where $\f'$ is obtained from 
$\f\otimes \f_{nj} : \s^*(P_0)\otimes \s^*(P_j^{(n)}) \to P_0 \otimes P_j^{(n)}$. 
By 3.6,  
$F_j^{(n)}$ gives a functor $\SP_{\b} \mapsto \SP_{\b + n\a_j}$. 
It induces an operator $f_j^{(n)}: K(\SP_{\b}) \to K(\SP_{\b + n\a_j})$.
\par
On the other hand, assume that $\b \in Q_+^{\s}$ such that $\b- n\a_j \in Q_+^{\s}$ for 
$j \in J$. 
We define a functor $E_j^{(n)}: \SL_{\b} \to \SL_{\b - n\a_j}$ as follows; 
write $P(j)^{(n)} = (P_j^{(n)}, \f_{nj})$ as in 3.8, $P_j^{(n)}$ is the projective 
$R(n\a_j)$-module. Let $M = (M_0, \f) \in \SL_{\b}$, where 
$M_0 \in R(\b)\gmm$.  Set 
\begin{equation*}
\tag{3.9.3}
E_j^{(n)} : (M_0, \f) \mapsto 
      \bigl((R(\b-n\a_j)\circ P_j^{(n)})^{\psi}\otimes_{R(\b)}M_0, \f'\bigr)
       = (M_1, \f'), 
\end{equation*}
where $\f'$ is induced naturally from $\f$ and $\f_{nj}$. 
As in 2.6,
$M_0 \mapsto M_1$ is an exact functor.
Hence $E_j^{(n)}$ is an exact functor.  It induces an operator 
${e_j'}^{(n)}: K(\SL_{\b}) \to K(\SL_{\b - n\a_j})$. 
\par 
In the case where $n = 1$, 
for $M = (M_0, \f_0) \in \SC_{\b}$, $E_jM = (\wt E_jM_0, \f)$, where $\wt E_jM_0$
 is given by 
\begin{equation*}
\wt E_jM_0 = (R(\b - \a_j)\otimes_{\Bk} P_j^{\psi})\otimes_{R(\b - \a_j)\otimes R(\a_j)}
          e(\b - \a_j, \a_j)M_0.
\end{equation*}
Note that $P_j = R(\a_j)e(\nu)$ with $\nu = (i_1, \dots, i_t)$ by (3.7.3). 
Thus $P_j^{\psi} = e(\nu)R(\a_j)$, and so
\begin{equation*}
\tag{3.9.4}
\wt E_jM_0 = e(\b - \a_j, \a_{i_1}, \dots, \a_{i_t})M_0
       = E_{i_1}\cdots E_{i_t}M_0,
\end{equation*}
where  
$e(\b - \a_j, \a_{i_1}, \dots, \a_{i_t})$ is defined similarly to 
$e(\b,\g)$ in 2.2, and $E_{i_k}$ is the functor defined in (2.6.3)
\par
The following result is a generalization of (2.6.5).

\begin{lem}  
Assume that $j \in J$, $\b \in Q_+^{\s}$.  Then we have
\begin{align*}
\lpp f_j^{(n)}[P], [M] \rpp &= \lpp [P], e_j'^{(n)}[M]\rpp, 
        \qquad  (P \in \SP_{\b}, M \in \SL_{\b + n\a_j}). \\ 
\end{align*}
\end{lem}
\begin{proof}
We define functors
\begin{align*}
 \wt F_j^{(n)} &: R(\b)\gpp \to R(\b + n\a_j)\gpp, \quad 
P_0 \mapsto P_0 \circ P_j^{(n)}, \\
\wt E_j^{(n)} &: R(\b+ n\a_j)\gmm \to R(\b)\gmm, \quad
  M_0 \mapsto (R(\b)\circ P_j^{(n)})^{\psi}\otimes_{R(\b+n\a_j)}M_0.  
\end{align*} 

Let $P = (P_0, \f) \in \SP_{\b}$, $M = (M_0, \f') \in \SL_{\b + n\a_j}$. 
We have
\begin{align*}
\tag{3.10.1}
(\wt F_j^{(n)}&P_0)^{\psi}\otimes_{R(\b + n\a_j)}M_0   \\ 
            &\simeq \biggl(P_0^{\psi}\otimes_{R(\b)}(R(\b)
     \otimes_{\Bk}(P_j^{(n)})^{\psi})_{R(\b)\otimes R(n\a_j)}
        R(\b + n\a_j)\biggr)\otimes_{R(\b + n\a_j)}M_0 \\
              &= P_0^{\psi}\otimes_{R(\b)}\wt E_j^{(n)}M_0.
\end{align*}
Let $Z$ be the left hand side of (3.10.1). 
The isomorphisms $\f : \s^*P_0 \isom P_0, \f': \s^*M_0 \isom M_0$ induce 
the isomorphism $\s^*Z \isom Z$, which coincides with the isomorphism 
$\s^*(P_0^{\psi}\otimes_{R(\b)}\wt E_j^{(n)}M_0) 
    \isom P_0^{\psi}\otimes_{R(\b)}\wt E_j^{(n)}M_0$ induced from $\f$ and $\f'$, 
through the isomorphisms in (3.10.1).  
Thus the lemma is proved. 
\end{proof}

\para{3.11.}
As discussed in 3.6, $K(\SP)$ is an associative algebra over 
$\BZ[\z_{\Bn}, q, q\iv]$ with the basis $\ul\bB$. 
We define $\wt\BA$ as the smallest subring of $\BZ[\z_{\Bn}, q,q\iv]$
containing $\BA$  
such that all the structure constants of the algebra $K(\SP)$ with respect to 
the basis $\ul\bB$ lie in $\wt\BA$. 
\par
Let ${}_{\wt\BA}K(\SP)$ be the $\wt\BA$-submodule of $K(\SP)$ spanned 
by $\ul\bB$. Then ${}_{\wt\BA}K(\SP)$ is the $\wt\BA$-subalgebra of $K(\SP)$. 
Set ${}_{\wt\BA}\ul\BU_q^- = \wt\BA \otimes_{\BA}\ul\BU_q^-$.  
The following result was proved by McNamara [M]. 
\begin{thm}[{[M, Thm. 6.1]}]  
There exists a unique $\wt\BA$-algebra isomorphism 
\begin{equation*}
\tag{3.12.1}
\g_0 : {}_{\wt\BA}\ul\BU_q^- \isom {}_{\wt\BA}K(\SP)
\end{equation*}
such that $\g_0(f_j^{(n)}) = [P(j)^{(n)}]$. 
Moreover, $\g_0$ is an isometry. 
\end{thm}

The following result is a generalization of Proposition 2.9.

\begin{prop}  
\begin{enumerate}
\item 
Consider the actions of $f_j^{(n)}$ on $K(\SP)$ defined in (3.9.1), 
and on $\ul\BU_q^-$ by the left multiplication. 
Then there exists an isomorphism of $\wt\BA$-modules (actually 
anti-algebra isomorphism) 
\begin{equation*}
\tag{3.13.1}
\g : {}_{\wt\BA}\ul\BU_q^- \isom {}_{\wt\BA}K(\SP)
\end{equation*} 
such that $\g$ maps $1 \in {}_{\wt\BA}(\ul\BU_q^-)_0$ to 
$[\Bk, \id] \in {}_{\wt\BA}K(\SP_0)$, commuting with $f_j^{(n)}$.
Moreover, $\g$ is an isometry. 
\item  Consider the actions of $e_j'^{(n)}$ on $K(\SL)$ defined in (3.9.2), 
and on $(\ul\BU_q^-)^*$ as in 1.9.
Then there exists an isomorphism of 
$\wt\BA$-modules,
\begin{equation*}
\tag{3.13.2}
\g^* : {}_{\wt\BA}K(\SL) \isom {}_{\wt\BA}(\ul\BU_q^-)^*
\end{equation*} 
such that $\g^*$ maps $[\Bk,\id] \in {}_{\wt\BA}K(\SL_0)$ to $1^* \in {}_{\wt\BA}(\ul\BU_q^-)^*_0$, 
commuting with $e_j'^{(n)}$. 
\end{enumerate}
\end{prop}
\begin{proof}
The proof is similar to that of Proposition 2.9.
We define $\g : {}_{\wt\BA}\ul\BU_q^- \to {}_{\wt\BA}K(\SP)$ by 
$\g = \g_0 \circ *$, where $* : {}_{\wt\BA}\ul\BU_q \to {}_{\wt\BA}\ul\BU_q^-$
is the anti-involution. 
Since the action of $f_j^{(n)}$ on $K(\SP)$ is defined by (3.9.1), 
(i) follows from  Theorem 3.12. If we define $\g^*$ as the transpose of $\g$, 
(ii) follows from (i) by Lemma 3.10.
\end{proof}

\para{3.14.}
$f_j^{(n)} \in {}_{\wt\BA}\ul\BU_q^-$ satisfies the relation 
$[n]^!_jf_j^{(n)} = f_j^n$.   Hence by (3.13.1), the operator 
$f_j^{(n)}$ on ${}_{\wt\BA}K(\SP)$ satisfies a similar relation. 
Then by using Lemma 3.10, we obtain the corresponding relation on the operator ${e'}_j^{(n)}$
on ${}_{\wt\BA}K(\SL)$. Thus we have
\begin{equation*}
\tag{3.14.1}
[n]^!_jf_j^{(n)} = f_j^n, \qquad [n]^!_j{e'}_j^{(n)} = {e'}_j^n
\end{equation*} 
as operators on ${}_{\wt\BA}K(\SP)$ and on ${}_{\wt\BA}K(\SL)$. 

\para{3.15.}
In [M, Lemma 11.3], McNamara showed that 
$\wt\BA \subset \BZ[\z_{\Bn} + \z_{\Bn}\iv, q, q\iv]$, and paused a
question whether $\wt \BA$ coincides with $\BA$ or not ([M, Question 11.2]).
In this paper, we show that this certainly holds, namely, 

\begin{thm} 
The structure constants in $K(\SP)$ with respect to the basis 
$\ul\bB$ are all contained in $\BA$.  Thus the isomorphisms 
(3.12.1), (3.13.1) and (3.13.2) can be replaced by the isomorphisms
\begin{equation*}
\g_0, \g  : {}_{\BA}\ul\BU_q^- \isom {}_{\BA}K(\SP) 
\quad \text{ and} \quad \g^* : {}_{\BA}K(\SL) \isom {}_{\BA}(\ul\BU_q^-)^*.
\end{equation*}
\end{thm}

The theorem will be proved in the next section.

\par\bigskip
\section {The proof of Theorem 3.16 }

\para{4.1.}
We keep the setup in Section 3.  In particular, 
$\Bn$ is the order of $\s : I \to I$, and $\z_{\Bn}$ is a primitive 
$\Bn$-th root of unity in $\BC$. 
If $\Bn = 1$, Theorem 3.16 certainly holds by Theorem 2.7.  
We prove the theorem by induction on $\Bn$.  
Hence we assume that $\Bn > 1$, and
that the theorem holds for $\Bn' < \Bn$. 
We write $\Bn = \Bn_1\Bn_2$, where $\Bn_2 > 1$ is a power of a prime number 
$\ell$, and $\Bn_1$ is prime to $\Bn_2$. Choose $a, b \in \BZ$ such that 
$1 = \Bn_1a + \Bn_2b$, and set $\s_1 = \s^{\Bn_2b}, \s_2 = \s^{\Bn_1a}$. 
Thus $\s = \s_1\s_2$, and the order of $\s_1$ (resp. $\s_2$) is $\Bn_1$ (resp. $\Bn_2$). 
We define $\z_{\Bn_1} = (\z_{\Bn})^{\Bn_2b}$ and $\z_{\Bn_2} = (\z_{\Bn})^{\Bn_1a}$. 
Then $\z_{\Bn_1}$ (resp. $\z_{\Bn_2}$) is a primitive $\Bn_1$-th root 
(resp. $\Bn_2$-th root) of unity in $\BC$.  The ring homomorphism 
$\BZ[\z_{\Bn}] \to \Bk$ restricts to the ring homomorphism $\BZ[\z_{\Bn_1}] \to \Bk$, 
or to $\BZ[\z_{\Bn_2}] \to \Bk$.

\para{4.2.}
Assume that $\b \in Q_+^{\s}$. Let $\SC_{\b}$ be the abelian category 
associated to $\s^* : R(\b) \isom R(\b)$ as  
in 3.3, namely, the objects in $\SC_{\b}$ are $(M, \f)$, where $M$ is a graded 
$R(\b)$-module, and $\f : \s^*M \isom M$ satisfies the conditions in 3.1.   
Since $\s_1$ is an admissible diagram automorphism on $X$, and $\b$ is $\s_1$-invariant, 
one can define the abelian category $\SC'_{\b}$ by replacing $\s$ by $\s_1$. 
\par
For each $(M,\f) \in \SC'_{\b}$, we define an object 
$\s_2^*(M,\f) = (M', \f') \in \SC'_{\b}$ as follows; 
set $M' = \s_2^*M$, then $\s_2^*(\f) : \s_2^*\s_1^*M \isom \s_2^*M$ 
induces an isomorphism 
$\s_1^*M ' = \s_1^*\s_2^*M = \s_2^*\s_1^*M \isom \s_2^*M = M'$, 
which we denote by $\f': \s_1^*M ' \isom M '$.  
Clearly $\f' = \s_2^*(\f)$ satisfies the relation (3.1.1), and 
we obtain $(M',\f') \in \SC'_{\b}$. 
The functor $\s_2^* : \SC'_{\b} \to \SC'_{\b}$ is a periodic functor 
on the linear category $\SC'_{\b}$. 
Thus we define $\SC''_{\b}$ as the category associated to
the periodic functor $\s_2^*$ (see 3.1). 
\par
The category $\SC_{\b}''$ is described as follows;
objects are triples $(M, \f_1,\f_2)$ where $(M, \f_1) \in \SC'_{\b}$, and 
$\f_2 : \s_2^*(M,\f_1) \isom (M, \f_1)$ in $\SC'_{\b}$ satisfies the condition 
that $\f_2\circ \s_2^*\f_2\circ\cdots\circ (\s_2^*)^{\Bn_2-1}\f_2 = \id_{\SC'_{\b}}$.
A morphism from $(M,\f_1, \f_2)$ to $(M',\f_1', \f'_2)$ is defined by using 
a similar diagram as in (3.1.2) by replacing $M$ by $(M,\f_1)$ and $\s$ by 
$\s_2$. 
\par
By 3.3,
$\SC_{\b}$ is equivalent to the category of graded 
representations of the algebra 
$R(\b) \sharp (\BZ/\Bn\BZ)$.
Similarly, $\SC'_{\b}$ is equivalent to the category of graded representations 
of the algebra $R(\b)\sharp(\BZ/\Bn_1\BZ)$. 
By a similar argument, $\SC''_{\b}$ is equivalent to the category of graded 
representations of the algebra $\bigl(R(\b)\sharp(\BZ/\Bn_1\BZ)\bigr)\sharp(\BZ/\Bn_2\BZ)$.
Since $\Bk[\BZ/\Bn\BZ] \simeq \Bk[\BZ/\Bn_1\BZ] \otimes_{\Bk}\Bk[\BZ/\Bn_2\BZ]$, we see that
\begin{equation*}
\tag{4.2.1}
\bigl(R(\b)\sharp (\BZ/\Bn_1\BZ)\bigr)\sharp(\BZ/\Bn_2\BZ)
        \simeq R(\b)\sharp (\BZ/\Bn\BZ).
\end{equation*}
\par   
Assume given $(M, \f_1, \f_2) \in \SC''_{\b}$. 
$\f_2 : \s_2^*(M, \f_1) \isom (M,\f_1)$ induces an isomorphism $\s_2^*M \isom M$,   
which we also denote by $\f_2$. 
Then $\f_2\circ \s_2^*(\f_1): \s_2^*\s_1^* M \isom \s_2^*M  \isom M$ gives 
an isomorphism $\s^*M \isom M$, and  
the correspondence $(M, \f_1, \f_2) \mapsto (M, \f_2 \circ \s_2^*(\f_1))$ gives a 
functor $\SC_{\b}'' \to \SC_{\b}$.  By (4.2.1), this gives an equivalence 
of the categories $\SC_{\b}'' \isom \SC_{\b}$. 
\par
Recall that $\SP_{\b}$ is the full subcategory of finitely generated graded projective 
objects in $\SC_{\b}$, which is equivalent to the subcategory of 
finitely generated projective $R(\b)\sharp(\BZ/\Bn\BZ)$-modules.  Similarly, 
the full subcategory $\SP'_{\b}$ of $\SC'_{\b}$, and the full subcategory 
$\SP_{\b}''$ of $\SC''_{\b}$ are defined.   
Let $\SL_{\b}$ be the full subcategory of finite dimensional graded 
objects in $\SC_{\b}$, namely, the full subcategory of finite dimensional 
graded $R(\b)\sharp(\BZ/\Bn\BZ)$-modules.  The full subcategory $\SL'_{\b}$
of $\SC'_{\b}$, and the full subcategory $\SL_{\b}''$ of $\SC''_{\b}$
are defined similarly. 
\par
Summing up the above discussion, we obtain the following lemma. 
\begin{lem}  
The category $\SC''_{\b}$ (resp. $\SP''_{\b}$, $\SL''_{\b}$) 
is equivalent to the category $\SC_{\b}$ (resp. $\SP_{\b}$, $\SL_{\b}$). 
\end{lem}

\para{4.4.}
By the discussion in 3.2, 
if $(M, \f) \in \SC_{\b}$, then $(M, \z_{\Bn}\f) \in \SC_{\b}$, which is denoted by
$\z_{\Bn}(M, \f)$. Similarly, $\z_{\Bn_1}(M, \f_1)$ is defined for 
$(M, \f_1) \in \SC'_{\b}$. 
For a given $(M, \f_1, \f_2) \in \SC_{\b}''$, we have 
$(M, \z_{\Bn_1}\f_1, \z_{\Bn_2}\f_2) \in \SC_{\b}''$, which we denote by 
$\z_{\Bn}(M, \f_1, \f_2)$. 
Thus under the category equivalence $\SC_{\b}'' \isom \SC_{\b}$, 
the action of $\z_{\Bn}$ on $(M, \f_1,\f_2)$ corresponds 
to the action of that on $(M, \f_2\circ \s_2^*(\f_1))$. 
\par
The notion of traceless elements in $\SC_{\b}$, and 
the Grothendieck groups $K(\SP_{\b})$ and $K(\SL_{\b})$ were 
defined in 3.2 and 3.3.   This definition works also for 
$\SC'_{\b}$, and the Grothendieck group $K(\SP'_{\b}), K(\SL'_{\b})$
are defined. 
\par
The notion of traceless elements for $\SC''_{\b}$ 
is translated to our situation as follows. 
An object $(A,\f_1, \f_2) \in \SC''_{\b}$ is said to be traceless 
if there exists an object $(M, \f) \in \SC'_{\b}$, an integer $t_2 \ge 1$
which is a divisor of $\Bn_2$ such that $(\s_2^*)^{t_2}(M,\f) \isom (M,\f)$, 
and an isomorphism 
\begin{equation*}
\tag{4.4.1}
(A, \f_1) \simeq (M, \f) \oplus \s_2^*(M,\f) \oplus \cdots 
              \oplus (\s_2^*)^{t_2-1}(M, \f),
\end{equation*}  
where $\f_2 : \s_2^*(A, \f_1) \isom (A, \f_1)$ 
is defined as in 3.2. Furthermore, we assume that 
there exists an $R(\b)$-module $M_1$, an integer $t_1 \ge 1$ which is 
a divisor of $\Bn_1$ such that $(\s_1^*)^{t_1}M_1 \isom M_1$, 
and an isomorphism 
\begin{equation*}
\tag{4.4.2}
M \simeq M_1 \oplus (\s_1^*)M_1 \oplus \cdots\oplus (\s_1^*)^{t_1-1}M_1,
\end{equation*}   
where $\f : \s_1^*M \isom M$ is defined as in 3.2.  
Finally, we assume that $t_1t_2 \ge 2$. 
\par
Under the equivalence $\SC''_{\b} \simeq \SC_{\b}$, the traceless elements 
in $\SC''_{\b}$ correspond to the traceless elements in $\SC_{\b}$. 
We define the Grothendieck group $K(\SC''_{\b})$ in a similar way as in 3.2, 
by using th action of $\z_{\Bn}$, and the traceless elements as above. 
Then $K(\SC''_{\b}) \simeq K(\SC_{\b})$.  The Grothendieck groups 
$K(\SP''_{\b}), K(\SL''_{\b})$ are defined similarly, and we have
\begin{equation*}
\tag{4.4.3}
K(\SP''_{\b}) \simeq K(\SP_{\b}),  \qquad K(\SL''_{\b}) \simeq K(\SL_{\b}).
\end{equation*}

\para{4.5.}
Recall that the functor $\BD : \SC_{\b} \to \SC_{\b}$ 
is given by $\BD (M,\f) = (\BD M, \BD (\f)\iv)$ with respect to 
the contravariant functor $\BD : R(\b)\Mod \to R(\b)\Mod$. 
Under the category equivalence $\SC''_{\b} \simeq \SC_{\b}$, this is given by 
\begin{equation*}
\tag{4.5.1}
\BD : (M, \f_1, \f_2) \mapsto (\BD M, \BD (\f_1)\iv, \BD(\f_2)\iv),
\end{equation*}
where $\BD(\f_2)$ is the dual of $\f_2 : \s_2^*(M,\f_1) \to (M, \f_1)$ in $\SC'_{\b}$, 
and it is also the dual of $\f_2 : \s_2^*M \to M$ in $R(\b)\Mod$. 
Similar results hold also for $\SP''_{\b} \simeq \SP_{\b}$, and $\SL''_{\b} \simeq \SL_{\b}$.
\par
Here we note a lemma.
\begin{lem}  
Assume that $(P,\f) \in \SP_{\b}$ is a projective 
indecomposable object, which is not traceless.  Then $P$ is a 
projective indecomposable $R(\b)$-module. 
\end{lem} 

\begin{proof}
Let $(P, \f) \in \SP_{\b}$ be a projective indecomposable object, 
where $\f : \s^*P\isom P$. We assume that $(P,\f)$ is not traceless.  
Then $P$ is a projective $R(\b)$-module. 
Let $P_0$ be an indecomposable direct summand of $P$.  
Let $t$ be the smallest integer such that $(\s^*)^tP_0 \isom P_0$. 
Then $f\circ \s^*(\f)\circ \cdots\circ (\s^*)^{t-1}\f : (\s^*)^tP \isom P$ 
induces an isomorphism $(\s^*)^tP_0 \isom P_0$, which we denote by $\f'_0$.  
Set $M = P_0 \oplus \s^*P_0 \oplus\cdots \oplus (\s^*)^{t-1}P_0$.  
Then $\s^*M \isom M$, and the restriction of $\f : \s^*P \isom P$ 
on $\s^*M$ gives an isomorphism $\f_0 : \s^*M \isom M$, where 
$\f_0$ is defined by $\id : (\s^*)^kP_0 \isom (\s^*)^kP_0$ 
for $k = 1, \dots, t-1$, and $\f_0': (\s^*)^tP_0 \isom P_0$. 
Hence $(M, \f_0) \in \SP_{\b}$ is a direct summand of $(P, \f)$.  Since 
$(P,\f)$ is indecomposable, we have $(P,\f) = (M, \f_0)$. 
If $t \ge 2$, then $(M,\f_0)$ is traceless. Hence by assumption, 
$t = 1$, i.e., $P = P_0$ is indecomposable. The lemma holds. 
\end{proof}

\para{4.7.}
For a given $\b \in Q_+^{\s}$, 
let $\ul\bB_{\b}$ (resp. $\ul\bB_{\b}^*$) be the $\BZ[\z_{\Bn}, q,q\iv]$-basis of 
$K(\SP_{\b})$ (resp. $K(\SL_{\b})$) given in 3.4.  
Similarly, for a given $\b \in Q_+^{\s_1}$, 
$\BZ[\z_{\Bn_1}, q,q\iv]$-basis 
of $K(\SP'_{\b})$ (resp. $K(\SL'_{\b})$) 
is defined, which we denote by $\bB'_{\b}$ (resp. $\bB'^*_{\b}$).   
\par
Let $(P, \f, \f') \in \SP''_{\b}$ be a self-dual projective indecomposable object 
corresponding to an element in $\ul\bB_{\b}$ under the equivalence 
$\SP''_{\b} \simeq \SP_{\b}$.  Then $P \in R(\b)\gpp$ is projective indecomposable
by Lemma 4.6, and  $(P, \f) \in \SP'_{\b}$ is self-dual projective 
indecomposable by (4.5.1).  Thus we have  
\par\medskip\noindent
(4.7.1) \ The projection $(P,\f_1, \f_2) \mapsto (P,\f_1)$ 
induces a map $\ul\bB_{\b} \to \pm  \bB'_{\b}$. 
\par\medskip
Let $\b, \g \in Q_+^{\s}$. For
$(M, \f) \in \SC_{\b}, (N, \f') \in \SC_{\g}$, the convolution 
product is defined as 
$(M,\f) \circ (N,\f') = (M \circ N, \f \circ \f')$ as in (3.6.1), 
by using the convolution product $M \circ N \in R(\b+\g)\Mod$.
Under the equivalence $\SC''_{\b} \simeq \SC_{\b}$, this convolution 
product is expressed as
\begin{equation*}
\tag{4.7.2}
(M, \f_1, \f_2) \circ (N, \f_1', \f_2') 
     = (M\circ N, \f_1\circ \f_1',  \f_2\circ \f_2'),  
\end{equation*}
where $(M\circ N, \f_1\circ \f_1') = (M, \f_1) \circ (N, \f_1')$ applied 
for the category $\SC'$, and $\f_2\circ \f_2': \s_2^*(M\circ N) \isom M\circ N$
is defined similarly to $\f_1\circ \f_1': \s_1^*(M \circ N) \isom (M\circ N)$.  
\par
Similarly to $\wt\BA$ defined in 3.11, 
we define $\wt\BA_1$ as the smallest subring 
of $\BZ[\z_{\Bn_1}, q,q\iv]$ containing $\BA$ such that the structure constants with respect to 
$\bB'$ are contained in $\wt\BA_1$. 
Since $\Bn_1 < \Bn$, by induction hypothesis, we have $\wt \BA_1 = \BA$. 
\par
By (4.7.1) and (4.7.2), we have the following. 
\begin{lem}  
The structure constants 
of $\ul\bB$ in ${}_{\wt\BA}K(\SP)$ lie  in 
$\BZ[\z_{\Bn_2}, q, q\iv]$. 
\end{lem}

\para{4.9.}
Recall that 
$\ul X = (J, (\ ,\ ))$ is the Cartan datum induced from $(X, \s)$, and  
$\ul\BU_q^-$ is the quantum group associated to $\ul X$. 
Here $\s = \s_1\s_2$, with $\s_2$ : a power of $\ell$. 
Let $J_1$ be the set of $\s_1$-orbits in $I$, and 
$X _1 = (J_1, (\ ,\ ))$  the Cartan datum induced from $(X, \s_1)$.    
$\s_2$ acts on $J_1$, and induces an admissible diagram automorphism on $X_1$.
The set of $\s_2$-orbits on $J_1$ coincides with $J$, and the Cartan datum 
induced from $(X_1, \s_2)$ is isomorphic to $\ul X$.
We denote by $\BU_q'^-$ the quantum group associated to $X_1$.  
$\s_2$ acts on $\BU_q'^-$ as an automorphism. 
\par
Let ${}_{\BA}K(\SP')$ be the $\BA$-submodule of $K(\SP')$ spanned by 
$\bB'$.  Then by the induction hypothesis, ${}_{\BA}K(\SP')$ is 
the $\BA$-subalgebra of $K(\SP')$. 
By applying Theorem 3.16 to $K(\SP')$, we have an isomorphism of 
$\BA$-modules
\begin{equation*}
\tag{4.9.1}
\g_1 : {}_{\BA}\BU_q'^- \isom {}_{\BA}K(\SP'), 
\end{equation*}
which maps $1 \in {}_{\BA}(\BU_q'^-)_0$ to $[\Bk, \id] \in {}_{\BA}K(\SP'_0)$, 
and commutes with the actions of $f_{j_1}^{(n)}$ ($j_1 \in J_1, n \in \BN$).
\par
Let ${}_{\BA}K(\SL')$ be the $\BA$-submodule of $K(\SL')$ spanned by 
$\bB'^*$.  Then ${}_{\BA}K(\SL')$ is $\BA$-dual to ${}_{\BA}K(\SP')$, 
and $\bB'$ and $\bB'^*$ are dual to each other.  
Let $(\BU_q'^-)^*$ be the dual space of $\BU_q'^-$. 
By taking the transpose of $\g_1$, we have an isomorphism of $\BA$-modules,
\begin{equation*}
\tag{4.9.2}
\g_1^* : {}_{\BA}K(\SL') \isom {}_{\BA}(\BU_q'^-)^*,
\end{equation*}
which maps $[\Bk, \id] \in {}_{\BA}K(\SL'_0)$ to $1^* \in {}_{\BA}(\BU_q'^-)^*_0$, and 
commutes with the actions of $e_{j_1}'^{(n)}$ ($j \in J_1, n \in \BN$).  

\para{4.10.}
We consider the functor $\t = (\s_2\iv)^* : \SC'_{\b} \to \SC'_{\s_2(\b)}$ 
for any $\b \in Q_+^{\s_1}$. 
Note that the order of $\tau$ is equal to $\Bn_2$, which is a power of $\ell$.
Let $\BA' = \BF[q,q\iv]$ with $\BF = \BZ/\ell\BZ$. 
The discussion in Section 2 can be applied to our situation, by replacing 
$K_{\gp}$ (resp. $K_{\gm}$) by ${}_{\BA}K(\SP')$ (resp. ${}_{\BA}K(\SL')$) as follows.
Set ${}_{\BA'}K(\SP') = \BA' \otimes_{\BA}{}_{\BA}K(\SP')$, and 
define ${}_{\BA'}K(\SL')$ similarly. 
$\bB'$ (resp. $\bB'^*$) gives an $\BA'$-basis of $_{\BA'}K(\SP')$ 
(resp. ${}_{\BA'}K(\SL')$). 
Let ${}_{\BA'}K(\SP')^{\t}$ be the set of $\tau$-fixed elements in 
${}_{\BA'}K(\SP')$. Then 
$\{ O(b) \mid b \in \bB'\}$ gives an $\BA'$-basis of ${}_{\BA'}K(\SP')^{\tau}$.
Let $\CJ'$ be the subset of ${}_{\BA'}K(\SP')^{\tau}$ consisting of 
$O(x)$ for $x \in {}_{\BA'}K(\SP')$ such that $O(x) \ne x$.  
Then $\CJ'$ is an $\BA'$-submodule of ${}_{\BA'}K(\SP')^{\tau}$.
Similarly, $\BA'$-submodule $\CJ'^*$ of ${}_{\BA'}K(\SL')^{\tau}$ is defined.
We define the quotient modules $\CV_q'$ and $\CV'^*_q$ by 
\begin{equation*}
\tag{4.10.1}
\CV_q' = {}_{\BA'}K(\SP')^{\tau}/\CJ', \qquad 
    \CV'^*_q = {}_{\BA'}K(\SL')^{\tau}/\CJ'^*.
\end{equation*} 

\par
A $\s$-orbit $j$ in $I$ is regarded as a $\s_2$-orbit in $J_1$. 
Thus for each $j \in J$, one can write as $j = \{ j_1, \dots, j_t\}$ with $j_k \in J_1$. 
We define a functor $\wt F_j^{(n)}: \SP' \to \SP'$ by 
$\wt F_j^{(n)} = F_{j_t}^{(n)}\cdots F_{j_1}^{(n)}$. 
This induces an $\BA'$-homomorphism $\wt f_j^{(n)} = f_{j_t}^{(n)}\cdots f_{j_1}^{(n)}$
on ${}_{\BA'}K(\SP')$. 
As in 2.10, $\wt f_j^{(n)}$ commutes with $\tau$, and acts on $\CV'_q$.
\par
Similarly, for each $j \in J$, we define a functor 
$\wt E_j^{(n)} : \SL' \to \SL'$ by $\wt E_j^{(n)} = E_{j_1}^{(n)}\cdots E_{j_t}^{(n)}$.
This induces an $\BA'$-homomorphism $\wt e_j'^{(n)} = e_{j_1}^{(n)}\cdots e_{j_t}^{(n)}$
commuting with $\tau$. Thus we have an action of $\wt e_j'^{(n)}$ on $\CV_q'^*$.  
The pairing $\lp\ ,\ \rp : {}_{\BA}K(\SP') \times {}_{\BA}K(\SL') \to \BA$ 
induces a perfect pairing 
$\lp\ ,\ \rp : \CV_q' \times \CV_q'^* \to \BA'$, and $\wt e_j'^{(n)}$ 
coincides with the transpose of $\wt f_j^{(n)}$. 
\par
The quotient algebra $\BV_q' = {}_{\BA'}(\BU'^-_q)^{\s_2}/\BJ'$ 
of ${}_{\BA'}(\BU'^-_q)^{\s_2}$ is defined as in (1.6.1)  
by replacing $\s: \BU_q^- \to \BU_q^-$ by $\s_2 : \BU'^-_q \to \BU'^-_q$.   
Also the quotient module $\BV'^*_q = {}_{\BA'}(\BU'^-_q)^{*,\s_2}/\BJ'^*$ 
is defined as in (1.11.1). ($\BJ', \BJ'^*$ are the objects corresponding to 
$\BJ, \BJ^*$, respectively.) 
Note that the Cartan datum induced from $(X_1, \s_2)$ coincides with $\ul X$. 
Thus by Theorem 1.7 and Proposition 1.12, 
we see that there exist isomorphisms
\begin{equation*}
\tag{4.10.2}
\Phi_1 : {}_{\BA'}\ul\BU^-_q \isom \BV'_q, \qquad
\Phi_1^* : \BV'^*_q \isom {}_{\BA'}(\ul\BU^-_q)^*, 
\end{equation*}  
where under these isomorphisms, the actions of 
$\wt f^{(n)}_j$ (resp. $\wt e_j'^{(n)}$) 
corresponds to the actions of 
$f_j^{(n)}$ (resp. $e_j'^{(n)}$).  

\par
Since the isomorphism $\g_1 : {}_{\BA}\BU_q'^- \isom K(\SP')$ 
is compatible with the action of $\s_2$ and $\tau$, it induces an isomorphism 
${}_{\BA'}(\BU_q'^-)^{\s_2} \isom {}_{\BA'}K(\SP')^{\tau}$, which maps 
$\BJ'$ onto $\CJ'$.  
Similarly, the isomorphism 
$\g_1^* : K(\SL') \isom {}_{\BA}(\BU_q'^-)^*$ induces   
an isomorphism ${}_{\BA'}K(\SL')^{\tau} \isom {}_{\BA'}(\BU_q'^-)^{*,\s_2}$,
which maps $\CJ'^*$ onto $\BJ'^*$.
It follows that $\g_1$ and $\g_1^*$ induce isomorphisms 
\begin{equation*}
\tag{4.10.3}
\vT_1 : \BV_q' \isom \CV'_q, \qquad \vT_1^* : \CV_q'^* \isom \BV_q'^*, 
\end{equation*}  
where $\vT_1$ (resp. $\vT_1^*$ ) 
commutes with the actions of $\wt f_j^{(n)}$ 
(resp. $\wt e_j'^{(n)}$). 

\para{4.11.}
Let $\wt\BA$ be as in 3.11.  By Lemma 4.8, 
$\BA \subset \wt\BA \subset \BZ[\z_{\Bn_2}, q, q\iv]$.  We define 
\begin{equation*}
\wt\BA' = \BA'\otimes_{\BA}\wt\BA 
        \simeq \wt\BA /\ell \wt\BA. 
\end{equation*}
Then $\BA'$ is a subalgebra of $\wt\BA'$. 
\par
We compare the structure of ${}_{\wt\BA}K(\SL'')$ (resp. ${}_{\wt\BA}K(\SP'')$)
and $\CV'^*_q$ (resp. $\CV'_q$) by changing the base ring to $\wt\BA'$. 
Set ${}_{\wt\BA'}K(\SP') = \wt\BA'\otimes_{\BA'}{}_{\BA'}K(\SP')$, and 
${}_{\wt\BA'}K(\SL') = \wt\BA'\otimes_{\BA'}{}_{\BA'}K(\SL')$. 
The definition of $\CV'_q$ and $\CV'^*_q$ in (4.8.1) can be applied 
by replacing $\BA'$ by $\wt\BA'$, and one can define 
${}_{\wt\BA'}\CV'_q$ and ${}_{\wt\BA'}\CV_q'^*$ as the quotients of 
${}_{\wt\BA'}K(\SP')^{\tau}$ and ${}_{\wt\BA'}K(\SL')^{\tau}$, respectively, 
where $\CJ' \subset {}_{\wt\BA'}K(\SP')^{\tau}, \CJ'^* \subset {}_{\wt\BA'}K(\SL')^{\tau}$
are defined similarly.  
\par
Assume that $(M, \f_1, \f_2) \in \SL''_{\b}$ is a traceless element.
Then by 4.4, $(M,\f_1) \in \SL'_{\b}$ is traceless, or there exists 
$(M',\f') \in \SL'_{\b}$, and an integer $t \ge 2$ which is a divisor of $\Bn_2$, 
such that $(\s_2^*)^t(M',\f') \simeq (M',\f')$ and that  
\begin{equation*}
(M,\f_1) \simeq (M ',\f') \oplus \s_2^*(M ',\f') 
   \oplus \cdots \oplus (\s_2^*)^{t-1}(M ',\f'),
\end{equation*}
where $\f_2 : \s_2^*(M, \f_1) \isom (M,\f_1)$ is defined as in (4.4.1).
It follows that $[M,\f_1] = 0$ in $K(\SL'_{\b})$, or 
\begin{equation*}
[M, \f_1] = [M ',\f'] + \tau [M ', \f'] + \cdots + \tau^{t-1}[M ',\f'] \in \CJ'^*. 
\end{equation*} 
In particular, 
\par\medskip\noindent
(4.11.1) \ 
If $(M,\f_1, \f_2)$ is traceless, then $[M,\f_1] \in \CJ'^*$. 
\par\medskip

Let $(L, \f_1, \f_2) \in \SL''_{\b}$ be a simple object in $\SL''_{\b}$, 
which is not traceless, such that $[(L,\f_1, \f_2)] \in {}_{\wt\BA}K(\SL'')$. 
By [M, Thm. 3.6], $(L, \f_1)$ 
is a simple object in $\SL'_{\b}$ such that 
$\s_2^*(L,\f_1)  \isom  (L,\f_1)$. Here $L$ is a simple $R(\b)$-module, and 
is isomorphic to a self-dual simple module, up to degree shift. 
 By using the category equivalence
$\SL'' \simeq \SL$, 
there exist self-dual simple objects $(L, \f_1^{\bullet}, \f_2^{\bullet})$ in $\SL_{\b}''$.
Then $(L, \f_1^{\bullet})$ is self-dual in $\SL'_{\b}$, and there exists
a power $\z_1$ of $\z_{\Bn_1}$, with degree shift $q^a$, 
 such that $(L, \f_1^{\bullet}) = \z_1q^a(L, \f_1)$.  
By our assumption in 4.1, we may only consider the case where 
$\f_1^{\bullet} = \pm \f_1$, up to degree shift. 
Now if $\Bn_2$ is odd, 
$\f^{\bullet}_2 : \s_2^*(L,\f^{\bullet}_1) \isom (L,\f^{\bullet}_1)$ is unique, while 
if $\Bn_2$ is even, there exist exactly two $\pm \f^{\bullet}_2$ such that 
$(L, \f_1^{\bullet}, \f_2^{\bullet})$ is self-dual.  
Thus for $(L,\f_1, \f_2)$ as above, there exist $\z$ which is a power of $\z_{\Bn_2}$, 
with some degree shift $q^a$,
such that $(L, \f_1, \f_2) = \z q^a (L, \f_1^{\bullet}, \f_2^{\bullet})$. 
We attach $\z q^a[L, \f_1] \in {}_{\wt\BA'}K(\SL'_{\b})$ to 
$(L, \f_1, \f_2) \in \SL''_{\b}$. 
This gives a well-defined map since we are working on $\BF = \BZ/\ell\BZ$. 

\par
Take $(M,\f_1, \f_2) \in \SL_{\b}''$, and consider the composition series 
of $(M,\f_1,\f_2)$, where each composition factor is isomorphic to 
a simple object $(L,\f'_1, \f'_2) \in \SL''_{\b}$ 
such that $\s_2^*(L, \f'_1) \isom (L,\f'_1)$, 
or a traceless element. 
Accordingly, $(M,\f_1) \in \SC'_{\b}$ has a composition series, 
where each composition factor is isomorphic to  a simple object $(L, \f_1')$ 
such that $\s_2^*L \isom L$, or a traceless element in $\SL'_{\b}$.    
\par
In view of (4.11.1), we obtain  a well-defined $\wt\BA'$-module 
homomorphism 
\begin{equation*}
\tag{4.11.2}
\Psi^* : {}_{\wt\BA'}K(\SL'') \to {}_{\wt\BA'}\CV_q'^*,
\end{equation*}
where $[(L,\f_1,\f_2)]$ goes to $[\z q^a(L,\f_1)]$ for a simple module $L \in R(\b)\gmm$ 
such that $\s^*L \isom L$.  
By considering the transpose of (4.11.2), one can define an $\wt\BA'$-module 
homomorphism 
\begin{equation*}
\tag{4.11.3}
\Psi : {}_{\wt\BA'}\CV_q' \to {}_{\wt\BA'}K(\SP'').
\end{equation*}

Under the isomorphisms $K(\SL) \simeq K(\SL'')$, 
$K(\SP) \simeq L(\SP'')$, $e_j'^{(n)}$ (resp, $f_j^{(n)}$ ) 
acts on ${}_{\wt\BA'}K(\SL'')$ (resp. on  ${}_{\wt\BA'}K(\SP'')$).  
\begin{lem}  
\begin{enumerate}
\item The map $\Psi^*$ commutes with the action of
$e_j'^{(n)}$ on ${}_{\wt\BA'}K(\SL'')$, 
and that of $\wt e_j'^{(n)}$ on ${}_{\wt \BA'}\CV_q'^*$.
\item The map $\Psi$ commutes with the action of $f_j^{(n)}$ on 
${}_{\wt\BA'}K(\SP'')$, and that of $\wt f_j^{(n)}$ 
on ${}_{\wt\BA'}\CV_q'$.  
\end{enumerate}
\end{lem}
\begin{proof}
(ii) follows from (i) by the discussion in 4.10, and by Lemma 3.10.
\par
We prove (i). First consider the case where $n = 1$.  Assume that $\b \in Q_+^{\s}$ 
such that $\b - \a_j \in Q_+^{\s}$.  
Take $(M,\f) \in \SL_{\b}$, which corresponds to $(M, \f_1, \f_2) \in \SL''_{\b}$
under the equivalence $\SL'' _{\b}\simeq \SL_{\b}$. 
By the discussion in 3.9, $E_j(M,\f)$ is written as 
$E_j(M, \f) = (\wt E_jM, \f')$, where $\wt E_j M$ is given as in (3.9.4), and 
$\f'$ is determined uniquely from $\f$.  Thus $E_j(M, \f_1, \f_2)$ is written as
$E_j(M, \f_1, \f_2) = (\wt E_j(M, \f_1), \f_2')$, where 
$\wt E_j(M,\f_1) \in \SL'_{\b - \a_j}$ 
is defined as in 4.10, and $\f_2'$ is determined uniquely from 
$\f_1, \f_2$.  
The functor $(M,\f_1) \mapsto \wt E_j(M,\f_1)$ induces the operation 
$\wt e'_j : K(\SL'_{\b}) \to K(\SL'_{\b-\a_j})$, and $\wt e'_j([M,\f_1])$ is 
$\tau$-invariant. Hence it induces the operation $\wt e'_j$ on ${}_{\wt\BA'}\CV_q'^*$. 
This shows that the map $\Psi^*$ commutes with the action of $e'_j$ on 
${}_{\wt\BA'}K(\SL'')$ and that of $\wt e'_j$ on ${}_{\wt\BA'}\CV_q'^*$. 
\par
We have $[n]^!_j{e'}^{(n)}_j = {e'}^n_j$ on ${}_{\wt\BA'}K(\SL'')$ by (3.14.1). 
On the other hand,  we have 
$\wt{e'}^{(n)}_j = {e'}^{(n)}_{j_t}\cdots {e'}^{(n)}_{j_1}$ on ${}_{\wt\BA'}K(\SL')$.
Note that $[n]^!_{j_k}{e'}^{(n)}_{j_k} = {e'}^n_{j_k}$,  
where $q_j = q^{(\a_j,\a_j)/2} = q^{t(\a_i,\a_i)/2} = q_i^t$ for $i = j_k$. 
Hence we have 
\begin{equation*}
([n]^!_i)^t\wt{e'}^{(n)}_j = \wt{e'}_j^n. 
\end{equation*} 
Since $([n]^!_i)^t = [n]^!_j$ in $\wt\BA'$, and ${}_{\wt\BA'}\CV_q'^*$ is 
a free $\wt\BA'$-module, the assertion holds for any $n$. Thus (i) is proved. 
\end{proof}

\begin{prop}  
\begin{enumerate}
\item
The map $\Psi: {}_{\wt\BA'}\CV_q' \to {}_{\wt\BA'}K(\SP'')$ 
gives an isomorphism  of $\wt\BA'$-modules.
Furthermore the isomorphism $\g : {}_{\wt\BA'}\ul\BU_q^- \isom {}_{\wt\BA'}K(\SP)$ 
factors through isomorphisms
\begin{equation*}
\tag{4.13.1}
\xymatrix@C=36pt@ R=30pt@ M =8pt{
\g : {}_{\wt\BA'}\ul\BU_q^-   \ar[r]^-{\Phi_1} 
       &  {}_{\wt\BA'}\BV'_q \ar[r]^-{\vT_1}  
     & {}_{\wt\BA'}\CV'_q \ar[r]^-{\Psi}  &  {}_{\wt\BA'}K(\SP), 
}
\end{equation*}
where $\Phi_1, \vT_1, \Psi$ commute with the action of $f_j^{(n)}$
or $\wt f_j^{(n)}$. 
\item
The map $\Psi^* : {}_{\wt\BA'}K(\SL'') \to {}_{\wt\BA'}\CV_q'^*$ gives 
an isomorphism of $\wt\BA'$-modules. 
Furthermore the isomorphism $\g^* : {}_{\wt\BA'}K(\SL) \isom {}_{\wt\BA'}(\ul\BU_q^-)^*$
factors through isomorphisms
\begin{equation*}
\tag{4.13.2}
\xymatrix@C=36pt@ R=30pt@ M =8pt{
\g^* : {}_{\wt\BA'}K(\SL'')   \ar[r]^-{\Psi^*} 
       &  {}_{\wt\BA'}\CV'^*_q \ar[r]^-{\vT^*_1}  
     & {}_{\wt\BA'}\BV'^*_q \ar[r]^-{\Phi_1^*}  &  {}_{\wt\BA'}(\ul\BU_q^-)^*, 
}
\end{equation*}
where $\Psi^*, \vT_1^*, \Phi_1^*$ commute with the action of 
$e_j'^{(n)}$ or $\wt e_j'^{(n)}$.
\end{enumerate}
\end{prop}
\begin{proof}
We define a map $\wt\g : {}_{\wt\BA'}\ul\BU_q^- \to {}_{\wt\BA'}K(\SP'')$
as the composite of $\Phi_1, \vT_1$ and $\Psi$, 
\begin{equation*}
\xymatrix@C=36pt@ R=30pt@ M =8pt{
\wt\g : {}_{\wt\BA'}\ul\BU_q^-   \ar[r]^-{\Phi_1} 
       &  {}_{\wt\BA'}\BV'_q \ar[r]^-{\vT_1}  
     & {}_{\wt\BA'}\CV'_q \ar[r]^-{\Psi}  &  {}_{\wt\BA'}K(\SP''). 
}
\end{equation*}
By (4.10.2), the action of $f_j^{(n)}$ on ${}_{\wt\BA'}\ul\BU_q^-$ and that of
$\wt f_j^{(n)}$ on ${}_{\wt\BA'}\BV_q'$ are compatible with $\Phi_1$. 
By (4.10.3), the action of $\wt f_j^{(n)}$ is compatible with $\vT_1$.  
By Lemma 4.12, the action of $\wt f_j^{(n)}$ and that of 
$f_j^{(n)}$ are 
compatible with $\Psi$. It follows that $\wt\g$ is compatible with the action 
of $f_j^{(n)}$.  Moreover, $\wt\g$ maps $1 \in {}_{\wt\BA'}\ul\BU_q^-$ to 
$[(\Bk, \id, \id)] \in {}_{\wt\BA}K(\SP''_0)$.
On the other hand, by Proposition 3.13, we have an anti-algebra isomorphism 
$\g : {}_{\wt\BA'}\ul\BU_q^- \isom {}_{\wt\BA'}K(\SP)$, which commutes with 
the action of $f_j^{(n)}$,
under the identification $K(\SP) \simeq K(\SP'')$.  
Since $\g$ is determined uniquely by the action of $f_j^{(n)}$, we conclude that 
$\g = \wt\g$. 
In particular, $\Psi$ is an isomoprhism, and (i) is proved. 
(ii) follows from (i).  The proposition is proved. 
\end{proof}

\para{4.14.}
We are now ready to prove Theorem 3.16. 
Let $\ul\bB$ be the basis of ${}_{\wt\BA}K(\SP)$ given in 3.4, 
and $\ul\bB^*$ be the dual basis of ${}_{\wt\BA}K(\SL)$. 
Similarly, we obtain the basis $\bB'$ of $K(\SP')$, and its dual basis 
$\bB'^*$ of $K(\SL')$ over a suitable extension of $\BA$.   
But by our assumption in 4.1, we may assume that $\bB'$ is a basis of 
${}_{\BA}K(\SP')$, and $\bB'^*$ is the dual basis of ${}_{\BA}K(\SL')$. 
Let $(B')^{\tau}$ (resp. $(\bB'^*)^{\tau}$ be the set of $\tau$-fixed elements in 
$\bB'$ (resp. in $\bB'^*$).  Then the image $\ul\bB'$ of $(\bB')^{\tau}$ in $\CV_q$ 
gives an $\BA'$-basis of $\CV_q'$, and the image $\ul\bB'^*$ 
of $(\bB'^*)^{\tau}$ gives an $\BA'$-basis of $\CV_q'^*$.   
Now by the construction of the map 
$\Psi^* : {}_{\wt\BA'}K(\SL'') \to {}_{\wt\BA'}\CV_q'^*$, and by Proposition 4.13, 
$\Psi^*(\ul\bB^*)$ coincides with $\ul \bB'^*$.  
Hence $\Psi\iv (\ul\bB)$ coincides with $\ul\bB'$. 
In the diagram (4.13.1), the maps $\Phi_1, \vT_1$ are defined over $\BA'$. 
Thus $\g\iv(\ul\bB)$ gives an $\BA'$-basis of ${}_{\BA'}\ul\BU_q^-$. 
\par
On the other hand, since originally $\g$ is an isomorphism 
${}_{\wt\BA}\ul\BU_q^- \isom {}_{\wt\BA}K(\SP)$, we see that 
$\g\iv (\ul\bB)$ is an $\wt\BA$-basis of ${}_{\wt\BA}\ul\BU_q^-$.  
Fix $\b \in Q_-$, and let $Y$ (resp. $Y'$) 
be the $\wt\BA$-submodule (resp. $\BA$-submodule) of $\ul\BU_q^-$
spanned by $\g\iv(\ul\bB_{\b})$.  Then $Y/Y'$ is a finitely generated $\BZ$-module, 
and by the above discussion, $\ell(Y/Y') = 0$. Hence by 
Nakayama's lemma, $Y$ coincides with $Y'$ if the base ring is extended 
from $\BA = \BZ[q,q\iv]$ to $\BZ_{\ell}[q,q\iv]$, 
where $\BZ_{\ell}$ is the ring of $\ell$-adic integers. 
It follows that, for $b,b' \in \g\iv(\ul\bB)$, we have
$bb' = \sum_{b'' \in \g\iv(\ul\bB)}a_{bb'}^{b''}b''$, where
\begin{equation*}
a_{bb'}^{b''} \in \BZ[\z_{\Bn_2}, q, q\iv] \cap \BZ_{\ell}[q,q\iv] = \BZ[q, q\iv].
\end{equation*} 
Thus $\g\iv(\ul\bB)$ is an $\BA$-basis of ${}_{\BA}\BU_q^-$.  Since 
$\g$ is an anti-algebra isomorphism, we see that the structure constants 
of the basis $\ul\bB$ all lie in $\BA$, namely we have $\wt\BA = \BA$.
Theorem 3.16 is proved.  
\par\medskip
As a corollary to Theorem 3.16, a refinement of Proposition 4.13 
is obtained. 

\begin{cor}  
\begin{enumerate}
\item \ There exist isomorphisms of $\BA$-modules, 
\begin{equation*}
\g : {}_{\BA}\ul\BU_q^- \isom {}_{\BA}K(\SP), \qquad
\g^* : {}_{\BA}K(\SL) \isom {}_{\BA}(\ul\BU_q^-)^*. 
\end{equation*}
\item  \ The isomorphisms $\g, \g^*$ can be factored, as $\BA'$-modules, through  
\begin{align*}
\tag{4.15.1}
&\xymatrix@C=36pt@ R=30pt@ M =8pt{
\g : {}_{\BA'}\ul\BU_q^-   \ar[r]^-{\Phi_1} 
       &  {}_{\BA'}\BV'_q \ar[r]^-{\vT_1}  
     & {}_{\BA'}\CV'_q \ar[r]^-{\Psi}  &  {}_{\BA'}K(\SP), 
}  \\
\tag{4.15.2}
&\xymatrix@C=36pt@ R=30pt@ M =8pt{
\g^* : {}_{\BA'}K(\SL)   \ar[r]^-{\Psi^*} &  {}_{\BA'}\CV_q'^* \ar[r]^-{\vT_1^*}  
     & {}_{\BA'}\BV_q'^* \ar[r]^-{\Phi_1^*}  &  {}_{\BA'}(\ul\BU_q^-)^*. 
}
\end{align*}
\end{enumerate}
\end{cor}

\par\bigskip
\section{ KLR algebras and global bases }

\para{5.1.}
In this section, we assume that $R = \bigoplus_{\b \in Q_+}R(\b)$ is symmetric type, 
namely, the Cartan matrix $A = (a_{ij})$ is symmetric, and 
$Q_{i,j}(u,v)$ is a polynomial in $u - v$.  Furthermore, we assume that 
$\Bk$ is an algebraically closed field of characteristic 0. 
We consider the isomorphisms  $\g : {}_{\BA}\ul\BU_q^- \isom {}_{\BA}K(\SP)$
and $\g^*: {}_{\BA}K(\SL) \isom {}_{\BA}\ul\BU_q^-$ as in Corollary 4.15. 
In the last paragraph of [M], he states a result that $\g\iv(\ul\bB)$ coincides with 
the canonical basis, and also coincides with the lower global basis of $\ul\BU_q^-$, 
with a brief indication for the proof, which relies on the geometric argument. 
In this section we give a simple proof of this fact, without using the geometry.
\par
Let $\wt \vT_0: {}_{\BA}\BU_q^- \isom K_{\gp}$ be the isomorphism of 
$\BA$-algebras as given in Theorem 2.7. 
Recall that $\bB$ is the basis of $K_{\gp}$, and $\bB^*$ is the basis of $K_{\gm}$ 
given in 2.4.  
Let $\CB^*$ be the upper global basis of ${}_{\BA}(\BU_q^-)^*$ as given 
in Theorem 1.10, and $\CB$ the lower global basis of ${}_{\BA}\BU_q^-$.  
Also let $\CB\nat$ be the canonical basis of ${}_{\BA}\BU_q^-$.  
In the case where $X$ is symmetric, it is known by Grojnowski-Lusztig [GL]
that $\CB$ coincides with $\CB\nat$.  
\par 
The following result was proved by Varagnolo-Vasserot [VV], and Rouquier [R2], 
by using the Ext algebra obtained from  Lusztig's category $\CQ_V$
related to the geometry of quivers (see Introduction). 

\begin{thm} [{[VV], [R2]}]  
Under the setup in 5.1, the isomorphisms $\wt\vT_0 : {}_{\BA}\BU_q^- \isom K_{\gp}$  
sends $\CB\nat$ to $\bB$. 
\end{thm}

\para{5.3.}
Since $\CB\nat$ is invariant under $*$, $\CB\nat$ corresponds to $\bB$ 
under the isomorphism $\wt\vT : {}_{\BA}\BU_q^- \isom K_{\gp}$ 
in Proposition 2.9.  Since the basis $\CB\nat$ is almost orthonormal, 
$\bB$ is also almost orthonormal in the sense of 1.4, i.e., 
for $b, b' \in \bB$, 
\begin{equation*}
\tag{5.3.1}
(b,b') \in \begin{cases}
              1 + q\BZ[[q]]  &\quad\text{ if $b = b'$, } \\
              q\BZ[[q]]      &\quad\text{ if $b \ne b'$.}
           \end{cases}
\end{equation*}     

We want to show that $\ul\bB$ has a similar property. 

\begin{prop}  
Assume that $b \in \ul\bB$.  Then we have $\BD b = b$, and 
\begin{equation*}
\tag{5.4.1}
(b,b') \in \begin{cases}
              1 + q\BZ[[q]] &\quad\text{ if $b = b'$, }  \\
              q\BZ[[q]]     &\quad\text{ if $b \ne b'$. }
           \end{cases}
\end{equation*}
\end{prop}

\begin{proof}
Take $b, b' \in \ul\bB$, and write it as $b = [P], b' = [Q]$ for 
$P = (P_0,\f), Q = (Q_0,\f') \in \SP_{\b}$. The relation 
$\BD b = b$ is clear from the definition.
We show (5.4.1).  By Lemma 4.6, $P_0$ is an indecomposable projective 
$R(\b)$-module.  Since $(P_0,\f)$ is self-dual, $P_0$ is also self-dual. 
Thus $[P_0]$ corresponds to some element in $\bB$.  Hence by (5.3.1), 
$P_0$ satisfies the relation $([P_0], [P_0]) \in 1 + q\BZ[[q]]$.
We compare the inner product (2.5.3) for $K_{\gp}(\b)$ and the inner product 
(3.5.4) for $K(\SP_{\b})$. We show that
\begin{equation*}
\tag{5.4.2}
([P], [Q]) \in \begin{cases}
                1 + q\BZ[\z_{\Bn}][[q]] &\quad\text{ if $[P] = [Q]$, } \\
                q\BZ[\z_{\Bn}][[q]]      &\quad\text{ if $[P] \ne [Q]$. }
               \end{cases}
\end{equation*}
Since $\BD P_0 \simeq P_0$, by using the relation (3.5.2) (it also holds 
if $M$ is projective), in the definition of 
the inner product on $K_{\gp}(\b)$ and on $K(\SP_{\b})$, one can replace 
the tensor product $P_0^{\psi}\otimes_{R(\b)}P_0$ by the Hom space 
$\Hom_{R(\b)}(P_0,P_0)$. Then (5.3.1) shows that 
$\dim_{\Bk} \Hom_{R(\b)}(P_0, P_0)_0 = 1$, and 
$\dim_{\Bk} \Hom_{R(\b)}(P_0, q^kP_0)_0 = 0$ for
$k > 0$. In particular, the grading preserving homomorphisms $P_0 \to P_0$ are 
only the scalar multiplications.  
The action of $\f\otimes\f$ on $(P_0^{\psi}\otimes_{R(\b)}P_0)_0$ corresponds 
to the action $f \mapsto \f\iv \circ f \circ \f$ for $\Hom_{\BR}(P_0,P_0)_0$, 
hence it induces the identity map on $\Hom_{R(\b)}(P_0,P_0)_0$.  
This shows the first formula in (5.4.2).  
The second formula is obtained similarly from the relation 
$\dim \Hom_{R(\b)}(P_0, q^kQ_0) = 0$ for $k > 0$.   
\par
Now by Corollary 4.15 (i), we have an isometry
$\g: {}_{\BA}\ul\BU_q^- \isom  {}_{\BA}K(\SP)$. 
Hence $([P], [Q]) \in \BQ((q))$.  Combined with (5.4.2), 
we obtain (5.4.1). The proposition is proved.  
\end{proof}

\begin{cor}  
Let $\ul\CB$ be the lower global basis of ${}_{\BA}\ul\BU_q^-$, and 
$\ul\CB^*$ the upper global basis of ${}_{\BA}(\ul\BU_q^-)^*$.
Then $\g\iv(\ul\bB)$ coincides with $\ul\CB$ or with $\ul\CB\nat$, up to sign, 
and $\g^*(\ul\bB^*)$ coincides with $\ul\CB^*$, up to sign.
\end{cor}
\begin{proof}
By Proposition 5.4, $\g\iv(\ul\bB)$ gives an almost orthonormal basis 
in ${}_{\BA}\ul\BU_q^-$. Since $\g\iv(\ul\bB) \subset \wt\CB$, 
we have $\wt\CB = \g\iv(\ul\bB) \sqcup - \g\iv(\ul\bB)$. 
Hence $\g\iv(\ul\bB)$ coincides with $\ul\CB$ or with $\ul\CB\nat$, up to sign. 
The second statement 
follows easily from this.  
\end{proof}

\para{5.6.}
Take $j \in J$. Consider the operator $E_j : \SL_{\b} \to \SL_{\b -\a_j}$, 
for $\b \in Q_+^{\s}$ such that $\b-\a_j \in Q_+^{\s}$, defined 
as in (3.9.3) ($E_j = E_j^{(n)}$ with $n = 1$). 
$E_j$ induces an operator $e_j': K(\SL_{\b}) \to K(\SL_{\b - \a_j})$.
For $\b \in Q_+^{\s}$, we define a functor 
$F_j': \SL_{\b} \to \SL_{\b + \a_j}$ by 
\begin{equation*}
F_j' : M \mapsto M \circ L_j.
\end{equation*}
Since the convolution product is an exact functor, $F_j'$ is an exact 
functor. Thus it induces an operator $f_j : K(\SL_{\b}) \to K(\SL_{\b + \a_j})$.  
\par
For a simple object $M \in \SL_{\b}$, we set 

\begin{equation*}
\tag{5.6.1}
\ve_j(M) = \max\{ k \ge 0 \mid E_j^kM \ne 0\}.
\end{equation*} 
We define crystal operators 
$\wt\BE_j M \in \SL_{\b-\a_j}$, $\wt\BF_j M \in \SL_{\b + \a_j}$ 
by
\begin{equation*}
\tag{5.6.2}
\begin{aligned}
\wt\BE_j M &= q_j^{1 - \ve_j(M)}\soc E_jM, \\
\wt\BF_j M &= q_j^{\ve_j(M)}\head F'_jM.
\end{aligned}
\end{equation*}

If we use the restriction functor defined in 3.6, $E_jM$ is written as 
\begin{equation*}
E_jM \simeq \Hom_{R(\a_j)}(L(j), \Res_{\b - \a_j, \a_j}M).
\end{equation*}
Thus $\wt\BE_jM$ (resp. $\wt\BF_jM$) coincides with $\wt e_j^*M$ (resp. $\wt f_j^*M$) 
given in [M, 9],  where
\begin{align*}
\wt e^*_jM &= q_j^{1 - \ve_j(M)}\soc \Hom_{R(\a_j)}(L(j), \Res_{\b - \a_j,\a_j}M), \\
\wt f^*_jM &=  q_j^{\ve_j(M)}\head (M \circ L(j)).  
\end{align*} 
(Note that we use $\ve_j(M)$ as in (5.6.1).  This 
corresponds to $\ve_j^*(M)$ in the notation of [M].) 
\par
The following result is immediate from [M, Lemma 8.3], though 
this formula is not used in later discussions.

\begin{lem}  
Let $M \in \SL_{\b}$ be a simple object.
Then $\wt\BF_jM$ is a simple object such that 
$\ve_j(\wt\BF_jM) = \ve_j(M) +1$,  and $[F_j'M]$ is written 
in the Grothendieck group $K(\SL_{\b + \a_j})$, 
\begin{equation*}
[F_j'M] = q_j^{-\ve_j(M)}[\wt\BF_jM] + \sum_k[L_k],
\end{equation*}  
where $L_k \in \SL_{\b + \a_j}$ are simple objects such that 
$\ve_j(L_k) < \ve_j(M) + 1$. 
\end{lem}

In order to consider the action of $E_j$, we need a lemma, 
which is obtained from  Lemma 8.1 and Lemma 8.2 in [M].

\begin{lem}[{[M]}]  
\begin{enumerate}
\item  Let $N$ be a simple object in $\SC_{\b}$ such that $\ve_j(N) = 0$.
Set $M = N \circ L(j)^{(m)}$ for any $m \ge 1$. 
Then $\head M $ is irreducible, and $\ve_j(M )= m$. 
\item 
Let $N$ be a simple object in $\SC_{\b}$ such that $\ve_j(N) = m$.  Then 
there exists a simple object $X$ in $\SC_{\b - m\a_j}$ with $\ve_j(X) = 0$ such that
\begin{equation*}
\Res_{\b - m\a_j, m\a_j}(N) \simeq X \otimes L(j)^{(m)}.
\end{equation*}
\end{enumerate}
\end{lem}

We show the following.

\begin{prop}  
Let $M \in \SL_{\b}$ be a simple object such that $\ve_j(M) = m > 0$. 
Then $\wt\BE_j M$ is a self-dual simple object in $\SL_{\b - \a_j}$, and 
in the Grothendieck group 
$K(\SL)$, $[E_jM]$ is written as 
\begin{equation*}
\tag{5.9.1}
[E_jM] = G(q)[\wt\BE_jM] + \sum_{1 \le k \le p}[L_k],
\end{equation*}
where $G(q) \in \BN[q,q\iv]$, 
$\ve_j(\wt\BE_jM) = m - 1$, and  
$L_k \in \SL_{\b-\a_j}$ are simple objects such that $\ve_j(L_k) < m - 1$. 
\end{prop}

\begin{proof}
The proof is done by a similar argument as in the proof of [M, Lemma 9.4].
We give an outline of the proof below. 
It is known by [M, Lemma 9.4], $\wt\BE_jM$ is a self-dual simple 
object in $\SL_{\b-\a_j}$. We fix a composition series of $E_jM$, and let 
$N$ be a simple object such that $\ve_j(N) = m -1$, appearing as a composition 
factor. We show 
\par\medskip\noindent
(5.9.2) \ $N$ is isomorphic to $\wt\BE_jM$, up to degree shift.  
\par\medskip

Note that $E_jM \simeq \Hom_{R(\a_j)}(L(j), \Res_{\b - \a_j, \a_j}M)$.
Let $\wt N$ be a subobject of $E_jM$ such that $N$ is a simple quotient of $\wt N$.
By the adjunction isomorphism, we have

\begin{equation*}
\Hom_{\SC_{\b}}(\wt N \circ L(j), M) \simeq 
         \Hom_{\SC_{\b-\a_j}}(\wt N, \Hom_{R(\a_j)}(L(j), \Res_{\b -\a_j,\a_j}M)).
\end{equation*} 
Since the right hand side of this equality is non-zero, $M$ is a simple quotient 
of $\wt N \circ L(j)$. 
By Lemma 5.8 (ii), there exists a simple $X$ such that 
\begin{equation*}
\tag{5.9.3}
\Res_{\b -m\a_j, m\a_j}(M) \simeq X\otimes L(j)^{(m)}.
\end{equation*}
Since the restriction functor is exact, we have a surjective map
\begin{equation*}
\Res_{\b - m\a_j, m\a_j}(\wt N\circ L(j)) \to X \otimes L(j)^{(m)}.
\end{equation*}
We consider the Mackey filtration (see [M, Thm. 4.5]) for
\begin{equation*}
\Res_{\b-m\a_j,m\a_j}(\wt N \circ L(j))
      = \Res_{\b - m\a_j, m\a_j}\circ \Ind_{\b- \a_j, \a_j}(\wt N \otimes L(j))
\end{equation*}
Since $\ve_j(N) = m-1$, this Mackey filtration contains a non-zero factor
which involves $\Res_{\b-m\a_j, (m-1)\a_j}N$, up to degree shift. 
By Lemma 5.8 (ii), there exists a simple object $Y$ such that 
\begin{equation*}
\tag{5.9.4}
\Res_{\b - m\a_j, (m -1)\a_j}N \simeq Y \otimes L(j)^{(m -1)}. 
\end{equation*}
We have a surjective homomorphism $Y\otimes L(j)^{(m)} \to X \otimes L(j)^{(m)}$ with 
degree shift.
Since $X, Y$ are simple, we have $X \simeq Y$, up to degree shift.
\par
By considering the adjunction isomorphism in (5.9.4), we have a non-zero homomorphism 
$Y \circ L(j)^{(m-1)} \to N$.  
Since $\ve_j(N) = m-1$, we have $\ve_j(Y) = 0$. By Lemma 5.8 (i), 
$\head (Y \circ L(j)^{(m-1)})$ is irreducible.  It follows that 
$N \simeq \head (Y \circ L(j)^{(m -1)})$. 
By applying this argument to $\wt\BE_j M$ (the original setup in [M]), 
we see that $\wt\BE_jM \simeq \head (X \circ L(j)^{(m -1)})$.   
Hence $N \simeq \wt\BE_jM$, up to degree shift, and (5.9.2) holds. 
\par
Since $\ve_j(M) = m$, all the composition factors $L$ of $E_jM$ satisfy
the relation $\ve_j(L) \le m -1$. The proposition is proved. 
\end{proof}

\para{5.10.}
Let $\ul\bB = \bigsqcup_{\b \in Q_+^{\s}}\ul\bB_{\b}$ 
be the basis of $K(\SP)$, and $\ul\bB^* = \bigcup_{\b \in Q_+^{\s}}\ul\bB_{\b}^*$ 
be the basis of $K(\SL)$ 
defined as before.  The functors $\wt\BE_j$ and $\wt\BF_j$ induce 
operators on $K(\SL)$, which we denote by $\wt\Be_j$ and $\wt\Bf_j$. 
The following result was proved by 
Lemma 9.1, Lemma 9.3 and Lemma 9.4, together with the definition 
of $\ul\bB^*$ in [M].

\begin{prop}  
\begin{enumerate}
\item $\wt\Bf_j$ sends $\ul\bB^*_{\b}$ to $\ul\bB^*_{\b + \a_j}$.  
\item 
$\wt\Be_j$ sends $\ul\bB^*_{\b}$ to $\ul\bB^*_{\b -\a_j} \cup \{0 \}$. 
\item 
For $b, b' \in \ul\bB^*$, $b = \wt\Bf_jb'$ if and only if $b' = \wt\Be_jb$. 
\end{enumerate}
\end{prop} 

The following result is a generalization of Theorem 5.2. 
\begin{thm}  
Let $\g : {}_{\BA}\ul\BU_q^- \isom {}_{\BA}K(\SP)$, 
$\g^* : {}_{\BA}K(\SL) \isom {}_{\BA}(\ul\BU_q^-)^*$ be the 
isomorphisms as in Corollary 4.15.
\begin{enumerate}
\item $\g^*$ sends $\ul\bB^*$ to $\ul\CB^*$. 
\item $\g$ sends $\ul\CB$ to $\ul\bB$. 
\end{enumerate}
\end{thm}

\begin{proof}
We show (i). 
We prove, by induction on $|\b|$, that 
\begin{equation*}
\tag{5.12.1}
\g^*(\ul\bB^*_{\b}) = \ul\CB^*_{-\b} \qquad (\b \in Q_+^{\s}).
\end{equation*}
Since $\g^*$ maps $[1] = [\Bk,\id] \in {}_{\BA}K(\SL_0)$ to 
$1^* \in {}_{\BA}(\ul\BU_q^-)^*_0$, 
we have $\g^*(\ul\bB^*_0) = \ul\CB^*_0$.  Hence (5.12.1) holds for $\b = 0$.  Take 
$0 \ne \b \in Q_+^{\s}$, and assume 
that (5.12.1) holds for $\b' \in Q_+^{\s}$ such that $|\b'| < |\b|$.     
Let $b \in \ul\bB^*_{\b}$. There exists $j \in J$ such that $e_j'b \ne 0$. 
Then by Proposition 5.9 and by Proposition 5.11, 
$e_j'b$ can be written as
\begin{equation*}
\tag{5.12.2}
e_j'b = G(q)\wt \Be_j b + \sum_{b'}a_{bb'}b', \qquad (G(q) \in \BN[q,q\iv], a_{bb'} \in \BA),
\end{equation*}
where $\ve_j(\wt\Be_j b) = \ve_j(b) - 1$, 
and $b'$ runs over the elements in $\ul\bB^*_{\b-\a_j}$ such that $\ve_j(b') < \ve_j(b) -1$.  
Let $b_0 = \g^*(b)$.  By Corollary 5.5, $b_0 \in \ul\CB^*$, up to sign. 
Since the map $\g^*$ commutes with the action of $e_j'$, we have
\begin{equation*}
\tag{5.12.3}
e_j'(b_0) = G(q)\g^*(\wt \Be_j b) + \sum_{b'}a_{bb'}\g^*(b').
\end{equation*}
Here by induction, $\g^*(\wt\Be_j b), \g^*(b') \in \ul\CB^*_{-\b + \a_j}$. 
Moreover, since $\g^*$ commutes with $e_j'$, 
$\ve_j(\g^*(\wt\Be_j b)) = \ve_j(b_0) - 1$, and 
$\ve_j(\g^*(b')) < \ve_j(b_0) - 1$. 
Comparing (5.12.3) with Theorem 1.10 (iii) (applied for ${}_{\BA}(\ul\BU_q^-)^*$), 
we see that $b_0 = \g^*(b) \in \ul\CB^*_{-\b}$. 
This proves (5.12.1).  Hence (i) holds. (ii) follows from (i). 
\end{proof}

\remarks{5.13.}
(i) \ The quantum group of arbitrary type is obtained as $\ul\BU_q^-$ from
some $\BU_q^-$ of symmetric type, through an admissible automorphism $\s$. 
Hence our result shows that the global crystal basis of the quantum group of arbitrary
type is obtained from the KLR algebras through foldings.
\par
(ii) In general, the natural basis $\bB$ of $K_{\gp}$ has the positivity property, 
i.e., for $b, b' \in \bB$, $bb'$ is a linear combination of the basis with coefficients 
in $\BN[q,q\iv]$. This property does not hold for 
the global crystal basis $\CB$ of non-symmetric type.  Note that the basis $\ul\bB$ 
of $K(\SP)$ not necessarily has the positivity property.   

\para{5.14.}
By Theorem 5.2, the map $\wt\vT_0$ sends the canonical basis $\CB\nat$ to
$\bB$. By applying Theorem 5.12 for the case where $\s = \id$, we rediscover 
the result of [GL] that $\CB = \CB\nat$. In the following, 
we show that this also holds in the non-symmetric case.
\par
For $b \in \ul\bB, b' \in \ul\bB^*$ and $j \in J$, we set
\begin{align*}
\ve_j(b) &= \max\{ k \ge 0 \mid b \in f_j^kK(\SP) \}, \\
\ve_j(b') &= \max\{ k \ge 0 \mid e_j'^kb' \ne 0 \}. 
\end{align*}
Note that the definition $\ve_j(b')$ is compatible with the definition 
$\ve_j(M)$ in (5.6.1). 
\par
Take $b \in \ul\bB$, and let $b_* \in \ul\bB^*$ be the basis dual to $b$. 
Then we have 
\begin{equation*}
\tag{5.14.1}
\ve_j(b) = \ve_j(b_*).
\end{equation*}
\par
In fact, 
assume that $b \in f_j^nK(\SP)$.  Then 
$\lpp f_j^nK(\SP), b_*\rpp \ne 0$, and there exists $b' \in \ul\bB$ such that 
$\lpp f_j^nb', b_*\rpp = \lpp b', e_j'^nb_* \rpp \ne 0$, hence $e_j'^nb_* \ne 0$. 
On the contrary, assume that $e_j'^nb_* \ne 0$. Then there exists $b' \in \ul\bB$ such that
$\lpp b', e_j'^nb_*\rpp = \lpp f_j^nb', b_*\rpp \ne 0$. 
it is known ([L2, Thm. 14.3.2]) that a subset of $\CB\nat$ gives a basis of 
$f_j^n\ul\BU_q^-$ for each $n \ge 0$. Then by Corollary 5.5, 
a subset of $\ul\bB$ gives a basis of $f_j^nK(\SP)$ for each $n \ge 0$.  
It follows that $b \in f_j^nK(\SP)$. Thus (5.14.1) holds. 

\begin{thm}  
$\g\iv$ sends $\ul\bB$ to $\ul\CB\nat$.  Hence
the canonical basis $\ul\CB\nat$ and the global crystal basis $\ul\CB$ of 
$\ul\BU_q^-$ coincide with. 
\end{thm}

\begin{proof}
The proof is done just by applying some properties of $K(\SP)$ and $K(\SL)$, 
together with Theorem 5.12.
Below, we give an outline of the proof.
\par 
By Proposition 5.11, Theorem  5.12, together with Theorem 1.10 (iii), 
we have, for $b \in \ul\bB^*$ such that $\ve_j(b) = m > 0$, 

\begin{equation*}
\tag{5.15.1}
e_j'b = [\ve_j(b)]_j\wt\Be_j b + \sum_{b' \in \ul\bB^*}a_{bb'}b',  \qquad (a_{bb'} \in \BA), 
\end{equation*} 
where $\ve_j(\wt \Be_j b) = \ve_j(b) - 1$, and $a_{bb'} = 0$ unless 
$\ve_j(b') < \ve_j(b) -1$.  
\par
We define $\wt\Bf_j : \ul\bB_{\b-\a_j} \to \ul\bB_{\b}$ as the transpose 
of the map $\wt\Be_j : \ul\bB^*_{\b} \to \ul\bB^*_{\b-\a_j}$, i.e.,
$\wt\Bf_j(b) = b'$ if and only if $\wt\Be_j((b')_*) = b_*$, where $b_*, (b')_*$ are 
dual of $b, b'$, respectively. 
$\wt\Bf_j$ gives a bijection $\ul\bB_{\b-\a_j} \isom \ul\bB_{\b}$.
Since $\ul\bB$ is a dual basis of $\ul\bB^*$, (5.15.1) and (5.14.1) implies, 
for $b \in \ul\bB$,  that

\begin{equation*}
\tag{5.15.2}
f_jb = [\ve_j(b) + 1]_j\wt\Bf_jb + \sum_{b' \in \ul\bB}c_{bb'}b', \qquad (c_{bb'} \in \BA),
\end{equation*}    
where $\ve_j(\wt\Bf_jb) = \ve_j(b) +1$, and $c_{bb'} = 0$ unless 
$\ve_j(b') > \ve_j(b) + 1$. 
By (5.15.2), we have the following. 
\par\medskip\noindent
(5.15.3) \ Assume that $b \in \ul\bB$ is such that $\ve_j(b) = 0$. 
Then for any $n \ge 1$, there exists $b_1 \in \ul\bB$ such that 
\begin{equation*}
f_j^{(n)}b = b_1 + \sum_{b' \in \ul\bB}c_{bb'}b',  \qquad (c_{bb'} \in \BA),
 \end{equation*}
where $\ve_j(b_1) = n$, and $c_{bb'} = 0$ unless $\ve_j(b') > n$.  
Note that $c_{bb'} \in \BA$ follows from Corollary 4.15.
\par
For each $j \in J, n \in \BN$, set 
$\ul\bB_{n;j} = \{ b \in \ul\bB \mid \ve_j(b) = n\}$.
We define a map $\pi_{n;j} : \ul\bB_{0;j} \to \ul\bB_{n;j}$ by 
$b \mapsto b_1$ in (5.15.3).  $\pi_{n;j}$ is a bijection since $\wt\Bf_j$ is a bijection. 
\par
By using Proposition 5.11, we have 
\begin{equation*}
\bigcap_{j \in J}\{ b \in \ul\bB^* \mid \ve_j(b) = 0\} = [1] \in {}_{\BA}K(\SL_0)
\end{equation*}
This implies, for each $\b \in Q_+^{\s}$,  that 

\begin{equation*}
\tag{5.15.4}
\ul\bB_{\b} = \bigcup_{\substack{j \in J, n > 0 \\ \b - n\a_j \in Q_+^{\s} }}
                     \pi_{n;j}(\ul\bB_{0;j} \cap \ul\bB_{\b - n\a_j}).
\end{equation*}
We now consider the isomorphism $\g\iv : {}_{\BA}K(\SP) \isom {}_{\BA}\ul\BU_q^-$, 
and act $\g\iv$ on both sides of (5.15.4).  Since $\g$ commutes with the action of 
$f_j^{(n)}$, the right hand side of (5.15.4) corresponds exactly an inductive 
construction of canonical basis (see [L2, 14.4.2, Thm. 14.4.3]).
Thus $\g\iv(\ul\bB)$ coincides with the canonical basis $\ul\CB\nat$ 
of ${}_{\BA}\ul\BU_q^-$.    
The theorem now follows from Theorem 5.12.
\end{proof}

\remark{5.17.}
We have deduced the coincidence of $\ul\CB$ and $\ul\CB\nat$ by 
making use of the crystal structure of the basis $\ul\bB^*$ of $K(\SL)$.
However, if we assume the properties of canonical basis as in (5.15.4), 
then the coincidence $\ul\CB = \ul\CB\nat$ is shown,   
directly without appealing KLR algebras, by using a similar formula as
(5.15.3) for $\ul\CB$.   
   
\par

\par\vspace{1.5cm}
\noindent
Y. Ma \\
School of Mathematical Sciences, Tongji University \\ 
1239 Siping Road, Shanghai 200092, P.R. China  \\
E-mail: \verb|1631861@tongji.edu.cn|

\par\vspace{0.5cm}
\noindent
T. Shoji \\
School of Mathematica Sciences, Tongji University \\ 
1239 Siping Road, Shanghai 200092, P.R. China  \\
E-mail: \verb|shoji@tongji.edu.cn|

\par\vspace{0.5cm}
\noindent
Z. Zhou \\
School of Mathematical Sciences, Tongji University \\ 
1239 Siping Road, Shanghai 200092, P.R. China  \\
E-mail: \verb|forza2p2h0u@163.com|

\end{document}